\documentclass[10pt,reqno]{amsart}
\usepackage{amsrefs}
\usepackage[foot]{amsaddr}
\usepackage{amsmath,amssymb,multirow}
\usepackage{amsfonts,amssymb,amsthm}
\usepackage[pdftex]{graphicx}
\usepackage[tight]{subfigure}
\usepackage{tikz-cd}
\usepackage{color}

\newcommand{\p}{\mathbf{p}}
\newcommand{\D}{\mathbf{\Delta}}
\renewcommand{\L}{\mathbf{\Lambda}}
\renewcommand{\a}{\mathbf{a}}
\renewcommand{\i}{\mathbf{i}}
\renewcommand{\k}{\mathbf{k}}

\def\ie{\rm{i.e.\ }}
\def\eg{\rm{e.g.\ }}

\theoremstyle{plain}
\newtheorem{theorem}{Theorem} 
\newtheorem{lemma}{Lemma}
\newtheorem{proposition}{Proposition}

\theoremstyle{definition}
\newtheorem{definition}{Definition}

\numberwithin{equation}{section}

\begin{document}

\title{Exponential Random Simplicial Complexes}

\author[K. Zuev, O. Eisenberg, and D. Krioukov]{Konstantin Zuev, Or  Eisenberg, and Dmitri Krioukov}
\address{Department of Physics, Department of Mathematics, Department of Electrical\&Computer Engineering, Northeastern University, Boston, MA 02115, USA}
\email{k.zuev@neu.edu, eisenberg.o@husky.neu.edu, dima@neu.edu}

\begin{abstract}
Exponential random graph models have attracted significant research attention over the past decades. These models are maximum-entropy ensembles subject to the constraints that the expected values of a set of graph observables are equal to given values. Here we extend these maximum-entropy ensembles to random simplicial complexes, which are more adequate and versatile constructions to model complex systems in many applications. We show that many random simplicial complex models considered in the literature can be casted as maximum-entropy ensembles under certain constraints. We introduce and analyze the most general random simplicial complex ensemble $\D$ with statistically independent simplices. Our analysis is simplified by the observation that any distribution $\mathbb{P}(O)$ on any collection of objects $\mathcal{O}=\{O\}$, including graphs and simplicial complexes, is maximum-entropy subject to the constraint that the expected value of $-\ln \mathbb{P}(O)$ is equal to the entropy of the distribution. With the help of this observation, we prove that ensemble $\D$ is maximum-entropy subject to the  two types of constraints which fix the expected numbers of simplices and their boundaries.

\smallskip
\noindent \textbf{Keywords.} Random simplicial complexes, random graphs, maximum-entropy distributions, exponential random graphs model, network models.

\end{abstract}
\maketitle

\section{Introduction}

When studying complex systems consisting of many interconnected, interacting components, it is rather natural to represent the system as a graph or, more generally, as a simplicial complex. Modeling complex systems with graphs has proved to be useful for understanding systems as intricate as the Internet, the human brain, and interwoven social groups, and has led to a new area of research, called network science \cite{Newman10-book,Kleinberg10-book,Dorogovtsev10-book}.

A host of developed network models (\eg see \cite{Goldenberg10} for a survey) can be roughly divided into two classes: ``generative'' models and ``descriptive'' models \cite{ABFGXZ07}. Generative models are  algorithms which describe how to generate a network using some probabilistic rules for connecting nodes. These models primarily aim to uncover the hidden evolution mechanisms responsible for certain properties observed in real networks. A classical, and perhaps the simplest and best studied, example of a generative model is the Erd\H{o}s--R\'{e}nyi random graph $G(n,p)$ \cite{SolRap,ErRe59,ErRe60}: given $n$ nodes, place a link between every two nodes independently at random with probability $p$. Among other prominent examples are the preferential attachment model \cite{BarAlb99,KraReLe00,DoMeSa00} and the small-world model \cite{WatStr98,newmanwatts99,watts99} which explain the power-law degree distributions and small distances between most nodes, two universal properties observed in many real networks. Any generative model gives rise to an ensemble $(\mathcal{G},\mathbb{P})$, where $\mathcal{G}$ is a set of all graphs the model can possibly generate and $\mathbb{P}$ is the probability distribution on $\mathcal{G}$, where $\mathbb{P}(G)$ is the probability that the model generates $G\in\mathcal{G}$. One can always readily sample from $\mathbb{P}$ (using the network generating algorithm), but often cannot obtain a closed-form expression for $\mathbb{P}(G)$, or even implicitly describe $\mathbb{P}$  as a solution of some optimization problem equation.

Generative models can help to understand the fundamental organizing principles behind real networks and explain their qualitative behavior, but they are not specifically designed for network data analysis. Descriptive models attempt to fill this gap. A descriptive model is explicitly defined as an ensemble $(\mathcal{G}, \mathbb{P}_\theta)$, where $\mathcal{G}$ is a set of graphs and $\mathbb{P}_\theta$ is the joint probability distribution on $\mathcal{G}$ parameterized by a vector of parameters $\theta$, which are to be inferred from the observed network data. For any graph $G\in\mathcal{G}$, a descriptive model gives a closed-form expression for  $\mathbb{P}_\theta(G)$ which can be used for further statistical inference, \eg for estimating ensemble averages $\sum_{G\in\mathcal{G}}x(G)\mathbb{P}_\theta(G)$, where $x$ is a network property of interest.  In contrast to generative models, however, a descriptive model does not specify how to sample networks from $\mathbb{P}_\theta$, which is often a challenging task. In simple cases, a network model can be represented as both generative and descriptive model. For example, the Erd\H{o}s--R\'{e}nyi random graph $G(n,p)$ can be defined, as above, by a generative algorithm, or by the formula for the probability distribution $\mathbb{P}(G)=p^{f_1(G)}(1-p)^{{n\choose 2}-f_1(G)}$, where $f_1(G)$ is the number of edges in $G$.  In general, however, representing a generative model as descriptive (and vice versa) is a very difficult problem whose solution could be very useful for applications.

Exponential random graphs (ERGs)  \cite{HoLe81,FS86,strauss1986,PaNe04,Ko09,snijders2002markov}, often called $p^*$ models in the social network research community \cite{WaPa96,AnWaCr99,RoPaKaLu07}, are among the most popular and best studied descriptive models which provide a conceptual framework for statistical modeling of network data. Let $\mathcal{G}_n$ be the set of all simple graphs (without self-loops or multi-edges) with $n$ nodes, $x_1,\ldots, x_r$ be functions on $\mathcal{G}_n$, henceforth referred to as the graph observables, and let $\bar{x}_1,\ldots, \bar{x}_r$ be  the values of these observables $x_1(\bar{G}),\ldots,x_r(\bar{G})$ for a network of interest $\bar{G}\in\mathcal{G}_n$ computed from available network data. The ERG model defined by $\bar{G}$ and its observables $\bar{x}_1,\ldots, \bar{x}_r$ is the exact analog of the Boltzmann distribution in statistical mechanics:
\begin{equation}\label{eq:ERG}
\mathbb{P}_\theta(G)=\frac{e^{-H_\theta(G)}}{Z(\theta)},
\hspace{5mm} H_\theta(G)=\sum_{i=1}^r\theta_i x_i(G),
\end{equation}
where $H_\theta(G)$ is called the graph Hamiltonian, $Z(\theta)$ the partition function (the normalization constant), and $\theta=(\theta_1,\ldots,\theta_r)$ is a vector of model parameters which satisfy
\begin{equation}\label{eq:parameters}
-\frac{\partial\ln Z}{\partial\theta_i}=\bar{x}_i.
\end{equation}

Whereas originally (\ref{eq:ERG}) was simply postulated and used in empirical studies~\cite{HoLe81}, it was later recognized~\cite{strauss1986,FS86,PaNe04} that ERGs are maximum-entropy ensembles. Namely, the distribution defined by (\ref{eq:ERG}) and (\ref{eq:parameters}) maximizes the Gibbs entropy
\begin{equation}\label{eq:entropy}
S(\mathbb{P})=-\sum_{G\in\mathcal{G}_n}\mathbb{P}(G)\ln\mathbb{P}(G),
\end{equation}
subject to the $r$ ``soft''  constraints  and the normalization condition
 \begin{align}
 &\mathbb{E}_\mathbb{P}[x_i]=\sum_{G\in\mathcal{G}_n}x_i(G)\mathbb{P}(G)=\bar{x}_i, \label{eq:constraints1}\\
 &\sum_{G\in\mathcal{G}_n}\mathbb{P}(G)=1. \label{eq:constraints2}
 \end{align}
The general principle of maximum entropy is thoroughly reviewed in~\cite{Presse13}. In the context of complex networks, the principle of maximum entropy and different entropy measures are discussed in~\cite{AnBi09}. 
Despite some known problems with ERGs with nonlinearly correlated constraints~\cite{shalizi13,chatterjee13,HoCzTo14}, ERGs remain one of the most popular descriptive models for network data analysis, especially in social science.

In many cases, however, representing a complex system with a simplicial complex  --- a higher-dimensional analog of a graph --- is conceptually more sound than the basic network representation, and provides a ``higher order approximation'' of the system. Consider for example a social system of scientific collaboration. Three researchers may co-author a single article or they may have three different papers with two authors each. The network representation, where nodes are connected if the corresponding scientists co-authored a paper, will not distinguish between these two cases. But we can do this by placing (in the former case), or not (in the latter case), a 2-simplex on the three nodes. This is illustrated in Fig~\ref{fig1}. Other examples, where the simplicial complex representation is more accurate include biological protein-interaction systems, where proteins form protein complexes often consisting of more than two proteins, economic systems of financial transactions often involving several  parties, and social systems, where groups of people are united by a common motive, interest, or goal, as opposed to merely being pairwise connected.

 \begin{figure} \label{fig1}
 	\centerline{\includegraphics[width=100mm]{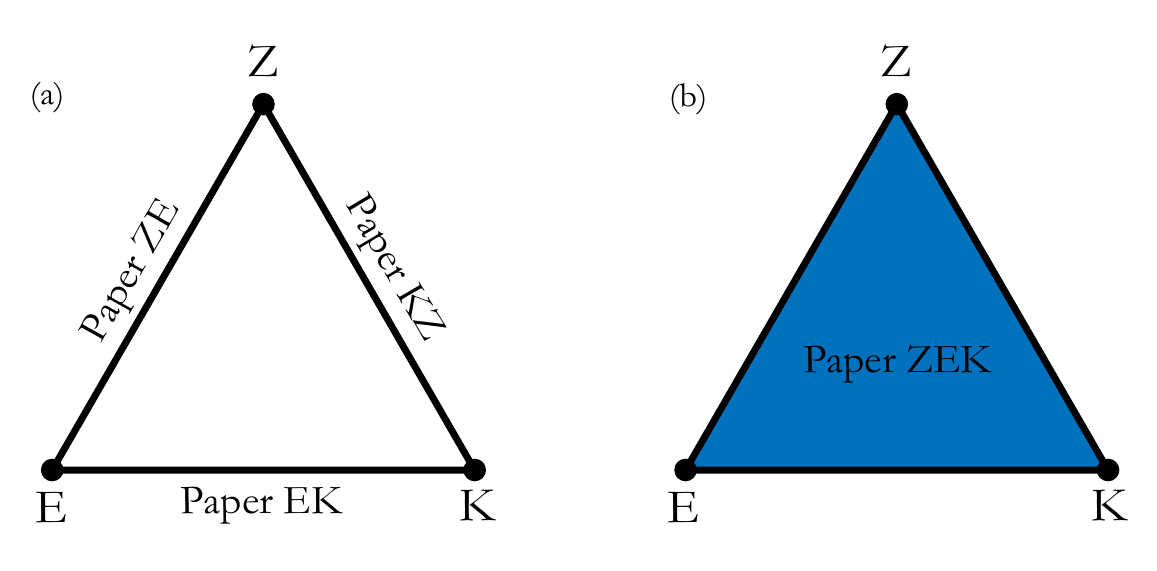}}
 	\caption{\small \textbf{Networks vs simplicial complexes.} The outcome of collaboration between scientists Z, E, and K could be three different papers co-authored by Z and E, E and K, and K and Z (Panel~(a)), or  a single paper co-authored by all three scientists (Panel~(b)). While the network representation does not distinguish between these two case and results in the graph in Panel~(a), the simplicial complex representation does by adding in the latter case the triangle $\{Z,E,K\}$ in Panel~(b). }
 \end{figure}

In general, compared to graphs, simplicial complexes encode more relevant information about a complex system, and make possible modeling beyond dyadic interactions. They have been used in many applications, including modeling social aggregation \cite{KeSpStMa13}, agent interaction \cite{Stasi14}, opinion formation and  dynamics \cite{MaRa09,MaRa14}, coverage and hole-detection in sensor networks \cite{Ghrist05}, and broadcasting in wireless networks \cite{Ren13}, to name just a few. We remark that prior to their being used for studying complex interactions, simplicial complexes were used in a rich variety of  geometric problems, ranging from grid generation in finite element analysis to modeling configuration spaces of dynamical systems \cite{Ed94}. Further details and applications can be found in \cite{EdHa10}.

In this paper, we introduce \textit{exponential random simplicial complexes} (ERSCs) which are higher dimensional generalizations of exponential random graphs, develop the formalism for ERSCs, and show that several popular generative models of random simplicial complexes  --- random flag complexes \cite{Kahle09}, Linial--Meshulam complexes \cite{LiMe06}, and Kahle's multi-parameter model \cite{Kahle14} --- can all be explicitly represented as ERSCs. We also introduce the most general ensemble of random simplicial complexes $\D$ with statistically independent simplices, and show that this ensemble is an ERSC ensemble as well.

\section{Basic Definitions and Notations}\label{notation}

Here we recall a few basic definitions and introduce  notation that we use throughout the paper.  For a comprehensive reference on simplicial complexes the reader is referred to \cite{Mu93}.

A \textit{simplicial complex} $C$ on $n$ vertices  $V=\{1,\ldots,n\}$ is a collection of non-empty subsets of $V$, called simplices. Complex $C$ contains all vertices, $\{i\}\in C$, and is closed under the subset relation: if $\sigma\in C$ and $\tau\subset\sigma$, then $\tau\in C$, where $\tau$ is called a \textit{face} of simplex $\sigma$, and $\sigma$ is a \textit{coface} of $\tau$. A simplex $\sigma$ is called a \textit{$k$-simplex} of dimension $k$ if its cardinality is $|\sigma|=k+1$.  It is useful to think of a $k$-simplex as the convex hull of $(k+1)$ points in general position in $\mathbb{R}^K$, $K\geq k$ \cite{Ha2002}. For instance, $0$-, $1$-, $2$-, and $3$-simplices are, respectively, vertices, edges, triangles, and tetrahedra. A simplicial complex is then a collection of simplices of different dimensions properly glued together. We say that $C$ has dimension $m$ if it has at least one $m$-simplex, but does not have simplices of higher dimension. Clearly, $m\leqslant n-1$.

\begin{figure} \label{fig2}
	\centerline{\includegraphics[width=80mm]{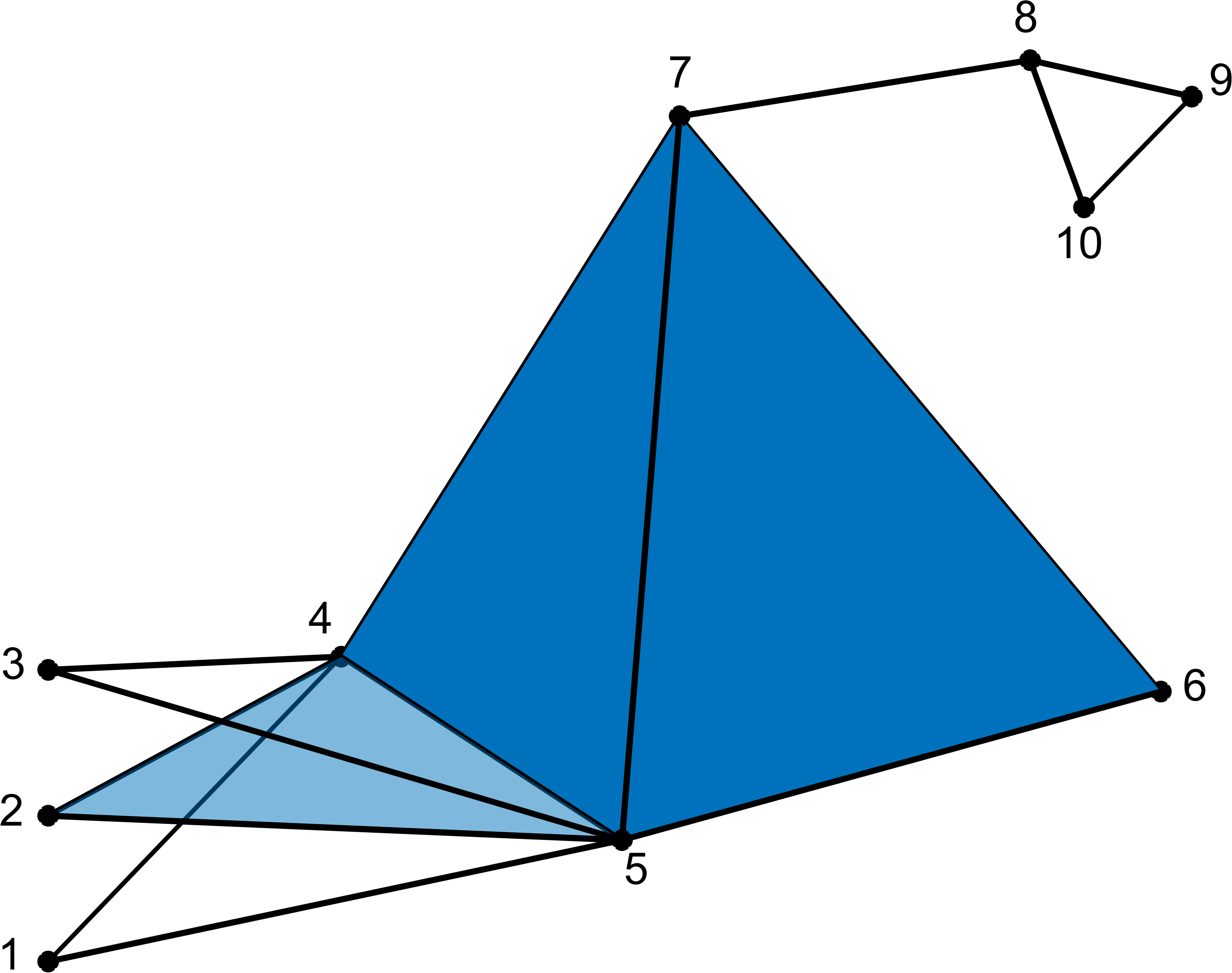}}
	\caption{\small \textbf{Simplicial complex and its adjacency tensors.} In this example, $C\in\mathcal{C}_{10}$, $\dim C=3$, and the non-zero elements $a_{\i_d}$ of adjacency tensor $\a_d$, $d=1,2,3,4,$ are: $a_{\i_1}=1$  for all  $\i_1=1,\ldots,10$; $a_{\i_2}=1$  for  $\i_2=(1,4), (1,5), (2,4), (2,5), (3,4), (3,5), (4,5), (4,6), (4,7), (5,6), (5,7)$, $(6,7), (7,8), (8,9), (8,10)$, and $(9,10)$; $a_{\i_3}=1$ for $\i_3=(2,4,5), (4,5,6), (4,5,7), (4,6,7)$, and $(5,6,7)$; $a_{\i_4}=1$ only for $\i_4=(4,5,6,7)$. The edge $\{4,6\}$ is not visible because of the 3-simplex $\{4,5,6,7\}$.}
\end{figure}

Let $\mathcal{C}_n$ be the set of all simplicial complexes on $n$ vertices. By analogy with graphs, where there exists a one-to-one correspondence between $\mathcal{G}_n$ and the set all boolean symmetric $n$-by-$n$ matrices with zeros on the diagonal, known as adjacency matrices, we can represent $\mathcal{C}_n$ by a tensor product
\begin{equation}
\mathcal{C}_n=\bigotimes_{d=1}^{n} \a_{d},
\end{equation}
where $\a_d=\{a_{i_1,\ldots,i_d}\}$, $i_j=1,\ldots,n$, $j=1,\ldots,d$, is a boolean symmetric tensor of order $d$ with zeros on all its diagonals. These conditions require precisely that $a_{i_1,\ldots,i_d}=a_{i_{\kappa(1)},\ldots,i_{\kappa(d)}}$ for any permutation $\kappa$ of subsubindices $1,\ldots,d$, and $a_{i_1,\ldots,i_d}=0$ if $i_j=i_k$ for any pair of $j$ and $k$. The non-redundant elements of tensor $\a_d$ are thus $a_{\i_d}$, where  multi-index $\i_d$ denotes a $d$-tuple of indices with increasing values:
 \begin{align}
 &\i_d=i_1,\ldots,i_d,\\
 &1\leqslant i_1<\ldots<i_d\leqslant n.
 \end{align}
The only requirement for $\bigotimes_{d=1}^{n} \a_{d}$ to be in bijection with $\mathcal{C}_n$ is then the following compatibility condition:
\begin{align}\label{condition}
a_{\i_d}&=1 \hspace{2mm}\Rightarrow \hspace{2mm} b_{\i_d}\stackrel{\tiny{\mathrm{def}}}{=}\prod_{k=1}^d a_{\i_d^{\hat{k}}}=1,\quad\text{where}\\
\i_d^{\hat{k}} &= i_1, \ldots ,\widehat{i_k}, \ldots ,i_d
\end{align}
is the $(d-1)$-long multi-index obtained from multi-index $\i_d$ by omitting index $i_k$. It is useful to think of $\i_d^{\hat{k}}$ as the result of operation $(\cdot)^{\hat{k}}$, which is the deletion of the $k^{\mathrm{th}}$ index, applied to multi-index $\i_d$.  Condition (\ref{condition}) simply formalizes the requirement that if the complex contains simplex $\{\i_d\}$, then it also contains all its faces.

For a simplicial  complex $C\in\mathcal{C}_n$,  $\a_d=\{a_{\i_d}\}$ is thus its ``adjacency'' tensor that  encodes the presence of $(d-1)$-simplices: $a_{\i_d}=1$ if  $\{\i_d\}\in C$, and zero otherwise. Since we assume that $C$ has $n$ vertices, we trivially have $\a_1={\bf 1}_n=(1,\ldots,1)$. Figure~\ref{fig2} illustrates the correspondence between simplicial complexes and their adjacency tensors.

A subcomplex of $C$ is a subset $C'\subset C$ that is also a simplicial complex. The \textit{$d$-skeleton} of $C$, denoted $C^{(d)}$, is a subcomplex consisting of all $k$-simplices of $C$ with $k\leq d$. The $1$-skeleton of a simplicial complex, for example, is a graph.
\begin{definition}
	The filled $d$-skeleton, denoted $C^{[d+1]}$, is a simplicial complex
	\begin{equation}
	C^{[d+1]}=C^{(d)}\cup\left\{\{\i_{d+2}\}: b_{\i_{d+2}}=1 \right\}.
	\end{equation}
	
\end{definition}

 In other words, $C^{[d+1]}$ is obtained from $C^{(d)}$ by adding $(d+1)$-simplices as follows. For every  $(d+1)$-simplex $\{\i_{d+2}\}$, if $C^{(d)}$ contains all $(d+2)$ $d$-simplices $\{\i_{d+2}^{\hat{k}}\}$, $k=1,\ldots,d+2$, we add  $\{\i_{d+2}\}$ to $C^{(d)}$. Intuitively,  we add $\{\i_{d+2}\}$ if its $d$-dimensional boundary  is already in $C^{(d)}$. Note that in this case we add $\{\i_{d+2}\}$ even if $\{\i_{d+2}\}\notin C$, and, therefore, $C^{[d+1]}$ is not necessarily a subcomplex of $C$. For example, $C^{[1]}$ is a complete graph on $n$ vertices, and $C^{[2]}$ is the $1$-skeleton of $C$ with all its triangular subgraphs filled by $2$-simplices.  We denote the filled $d$-skeleton by $C^{[d+1]}$ (instead of $C^{[d]}$), to emphasize that generally it has dimension $(d+1)$. Figure~\ref{fig3} illustrates the construction of a filled skeleton.
 \begin{figure}
    \centerline{\includegraphics[width=140mm]{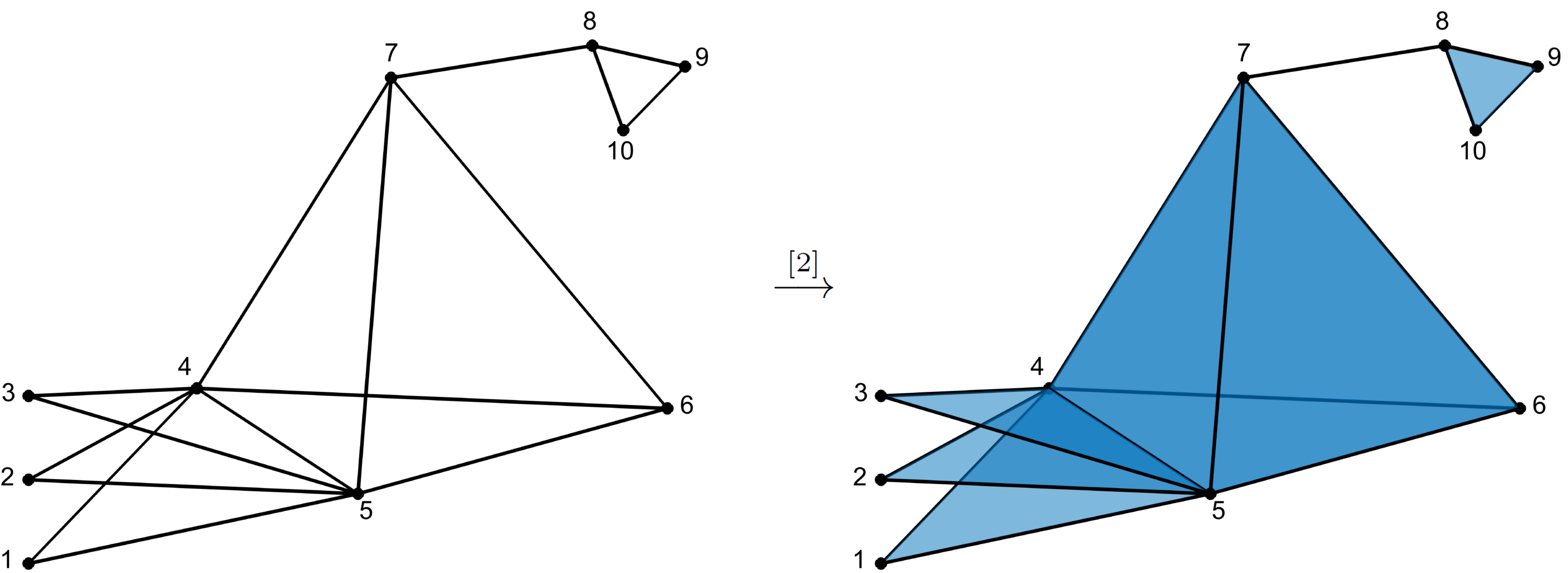}}
 	\caption{\small \textbf{Empty and filled skeletons.} The left simplicial complex is the $1$-skeleton $C^{(1)}$ of the complex $C$ in Fig.~\ref{fig2}. The filled $1$-skeleton $C^{[2]}$ on the right is obtained by adding all $2$-simplices based on triangular subgraphs of $C^{(1)}$. The 3-simplex $\{4,5,6,7\}\in C$ does not belong to $C^{[2]}$. The filled skeleton $C^{[2]}$ is not a subcomplex of $C$ since, for example, $\{8,9,10\}\notin C$. The filling operation is denoted by $\stackrel{{\footnotesize [2]}}{\longrightarrow}$.}
 	\label{fig3}
 \end{figure}

Thus, we have the following hierarchy of ``empty'' and ``filled'' skeletons:
\vspace{5pt} \\ \centerline{\includegraphics{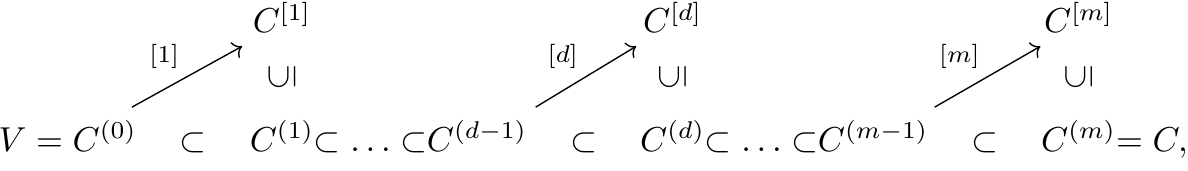}}
where $\stackrel{{\footnotesize [d]}}{\longrightarrow} $ denotes the filling operation.  Let $f_d$ denote the number of $d$-simplices in $C^{(d)}$ (and therefore in $C$), and $\phi_d$ be the number of $d$-simplices in $C^{[d]}$.
By construction, $\phi_d\geqslant f_{d}$, and
\begin{equation}
f_d=\sum_{\i_{d+1}}a_{\i_{d+1}} \hspace{3mm} \mbox{and} \hspace{3mm} \phi_d=\sum_{\i_{d+1}}b_{\i_{d+1}}.
\end{equation}
Figure~\ref{fig6} shows all simplicial complexes $C\in\mathcal{C}_3$ and the values of $f_1$, $f_2$, and $\phi_2$ for each $C$.

\section{Exponential Random Simplicial Complexes}

Let $\mathcal{S}$ be any subset of $\mathcal{C}_n$,  $\{x_1,\ldots,x_r\}$ be a set of functions on $\mathcal{S}$, $x_i:\mathcal{S}\rightarrow\mathbb{R}$, and $\{\bar{x}_1,\ldots,\bar{x}_r\}$ be a set of numbers, $\bar{x}_i\in\mathbb{R}$. We define the exponential random simplicial complex (ERSC) as a  maximum-entropy ensemble of complexes with ``soft'' constraints that require the observables $x_i$ to have the expected values $\bar{x}_i$ in the ensemble.

\begin{definition} An exponential random simplicial complex ERSC$(\mathcal{S},\{x_i\},\{\bar{x}_i\})$ is a pair $(\mathcal{S},\mathbb{P})$, where $\mathbb{P}$ is a probability distribution on $\mathcal{S}$ that maximizes the entropy
	\begin{equation}\label{eq:entropyC}
	S(\mathbb{P})=-\sum_{C\in\mathcal{S}}\mathbb{P}(C)\ln\mathbb{P}(C)\rightarrow \mbox{max},
	\end{equation}
	subject to the following constraints
	\begin{align}
	&\mathbb{E}_\mathbb{P}[x_i]=\sum_{C\in\mathcal{S}}x_i(C)\mathbb{P}(C)=\bar{x}_i \label{eq:constraints1C},\\
	&\sum_{C\in\mathcal{S}}\mathbb{P}(C)=1. \label{eq:constraints2C}
	\end{align}\label{def2}
\end{definition}

An exponential random simplicial complex is thus a descriptive model for random simplicial complexes. Generative models have been recently introduced and analyzed in~\cite{wu2015emergent,bianconi2015complex,bianconi2015complex2}.

 We can define ERSC for any  set of simplicial complexes, but, for most of the paper, we restrict ourselves to $\mathcal{C}_n$ and its subsets. If we use $\mathcal{S}=\mathcal{G}_n\subset\mathcal{C}_n$, then we recover the definition of ERGs. As with ERGs, the solution of the constrained optimization problem (\ref{eq:entropyC})-(\ref{eq:constraints2C}) belongs to the exponential family, hence the name of the ensemble.
\begin{theorem}\label{thm1} The maximum-entropy distribution $\mathbb{P}$ defined by (\ref{eq:entropyC})-(\ref{eq:constraints2C}) can be written as follows
\begin{equation}\label{eq:ERSC}
\mathbb{P}(C)=\frac{e^{-H(C)}}{Z(\theta)},
\hspace{5mm} H(C)=\sum_{i=1}^r\theta_i x_i(C), \hspace{5mm} Z(\theta)=\sum_{C\in\mathcal{S}}e^{-H(C)},
\end{equation}
where $H(C)$ is the Hamiltonian of simplicial complex $C\in\mathcal{S}$, $Z(\theta)$ is the normalizing constant, called the partition function, and $\theta=(\theta_1,\ldots,\theta_r)$ are the parameters satisfying the following system of $r$ equations
\begin{equation}\label{eq:parametersC}
-\frac{\partial\ln Z}{\partial\theta_i}=\bar{x}_i.
\end{equation}
\end{theorem}
The proof is nearly identical to the proof for ERGs \cite{PaNe04}, but we give it here for completeness.
\begin{proof} We use the standard method of Lagrange multipliers to solve the optimization problem (\ref{eq:entropyC})-(\ref{eq:constraints2C}).
Let $\theta_1,\ldots,\theta_r$ and $\alpha$ be the  Lagrange multipliers  for the constraints in (\ref{eq:constraints1C}) and (\ref{eq:constraints2C}). The Lagrangian is then
\begin{equation}
\mathcal{L}=-\sum_{C\in\mathcal{S}}\mathbb{P}(C)\ln\mathbb{P}(C)+\sum_{i=1}^r\theta_i\left(\bar{x}_i-\sum_{C\in\mathcal{S}}x_i(C)\mathbb{P}(C)\right)+\alpha\left(1-\sum_{C\in\mathcal{S}}\mathbb{P}(C)\right).
\end{equation}
The maximum entropy is achieved if the distribution $\mathbb{P}$ satisfies $\frac{\partial\mathcal{L}}{\partial\mathbb{P}(C)}=0$ for any $C\in\mathcal{S}$. This gives
\begin{equation}
-\ln\mathbb{P}(C)-1-\sum_{i=1}^r\theta_ix_i(C)-\alpha=0,
\end{equation}
or,
\begin{equation}
\mathbb{P}(C)\propto \exp\left({-\sum_{i=1}^r\theta_ix_i(C)}\right),
\end{equation}
which is equivalent to (\ref{eq:ERSC}), since $\sum_{C\in\mathcal{S}}\mathbb{P}(C)=1$. It remains to check that (\ref{eq:parametersC}) indeed holds:
\begin{equation}
\begin{split}
-\frac{\partial\ln Z}{\partial\theta_i}=&-\frac{1}{Z}\frac{\partial}{\partial\theta_i} \sum_{C\in\mathcal{S}}e^{-H(C)}=\frac{1}{Z}\sum_{C\in\mathcal{S}}\frac{\partial H(C)}{\partial \theta_i}e^{-H(C)}\\
=&\frac{1}{Z}\sum_{C\in\mathcal{S}}x_i(C)e^{-H(C)}=\sum_{C\in\mathcal{S}}x_i(C)\mathbb{P}(C)=\bar{x}_i,
\end{split}
\end{equation}
since the expected value of the observable $x_i$ in the ensemble is $\bar{x}_i$.
\end{proof}

\section{Simple Examples of ERSCs}\label{sec:simple examples}

Here we illustrate ERSCs with three simple examples: Erd\H{o}s--R\'{e}nyi random graphs $G(n,p)$, random flag complexes $X(n,p)$ and Linial--Meshulam random complexes $Y(n,p)$.

\subsection{Erd\H{o}s--R\'{e}nyi Random Graphs}

Perhaps the simplest nontrivial example of an ERSC is the Erd\H{o}s--R\'{e}nyi random graph ensemble $G(n,p)$, which can be viewed as a generative model for  $1$-dimensional simplicial complexes. $G(n,p)$ is a maximum-entropy ensemble with only one constraint that the expected number of edges $f_1$ in the ensemble is $ {n \choose 2}p$~\cite{PaNe04}:
\begin{equation}\label{eq:G=ERSC}
G(n,p)=\mathrm{ERSC}\left(\mathcal{G}_n,f_1,{n \choose 2}p\right).
\end{equation}

\subsection{Random Flag Complexes}

The flag complex $X(G)$ of a graph $G\in\mathcal{G}_n$, also called the clique complex or the Vietoris--Rips complex, is a (deterministic) simplicial complex in $\mathcal{C}_n$ whose $1$-skeleton is $G$ and whose $k$-simplices correspond to complete subgraphs of $G$, called cliques, of size $k+1$. Since any simplicial complex is homeomorphic to a flag complex, simplicial complexes arise in different applications and are often used for topological data analysis \cite{Zo12}.

Kahle \cite{Kahle09,Kahle14a} defines the random flag complex $X(n,p)$ as the flag complex of the Erd\H{o}s--R\'{e}nyi random graph,  $X(n,p)=X(G(n,p))$, and studies phase transitions of its homology groups. Here we show that $X(n,p)$ is, in fact, an ERSC.   \begin{proposition}\label{prop1}
Let $\mathcal{F}_n\subset\mathcal{C}_n$ be the set of all flag complexes on $n$ vertices, then
\begin{equation}\label{eq:X=ERSC}
X(n,p)=\mathrm{ERSC}\left(\mathcal{F}_n,f_1,{n \choose 2}p\right).
\end{equation}
\end{proposition}

Before giving the proof, we comment on what exactly  Proposition~\ref{prop1} states. $X(n,p)$ is a generative model of simplicial complexes: to generate $C\sim X(n,p)$, one first generates $G\sim G(n,p)$, and then sets $C=X(G)$. Let $\mathcal{S}_{X(n,p)}\subset\mathcal{F}_n$ denote the sample space of this random generative process, and $\mathbb{P}_{X(n,p)}$ be the resulting probability distribution on $\mathcal{S}_{X(n,p)}$. The random flag complex $X(n,p)$ can therefore be viewed as ensemble  $(\mathcal{S}_{X(n,p)},\mathbb{P}_{X(n,p)})$. Proposition~\ref{prop1} claims that $(\mathcal{S}_{X(n,p)},\mathbb{P}_{X(n,p)})$ is a maximum-entropy ensemble with $\mathcal{S}_{X(n,p)}=\mathcal{F}_n$ and a single constraint that the expected number of $1$-simplices is ${n\choose 2}p$. The proof is the same as for (\ref{eq:G=ERSC}), but we give it here for illustrative purposes.
\begin{proof}
First,  note that any flag complex $C\in\mathcal{F}_n$ can be generated by $X(n,p)$ with a non-zero probability:
\begin{equation}\label{eq:P_X}
\mathbb{P}_{X(n,p)}(C)=\mathbb{P}_{G(n,p)}(C^{(1)})=p^{f_1(C)}(1-p)^{{n\choose 2}-f_1(C)}.
\end{equation}
Therefore,  $\mathcal{S}_{X(n,p)}$ is indeed equal to $\mathcal{F}_n$. To prove (\ref{eq:X=ERSC}), we need to show that $\mathbb{P}_{X(n,p)}$ is in fact the ERSC probability distribution  (\ref{eq:ERSC}),(\ref{eq:parametersC}). Since every flag complex $C\in\mathcal{F}_n$ is completely defined by the adjacency matrix of its $1$-skeleton, $\mathcal{F}_n=\bigotimes_{d=1}^2\a_d=\mathbf{1}_n\bigotimes\a_2$. The partition function $Z$ can then be computed as follows:
\begin{equation}\label{eq:Z for X}
\begin{split}
Z(\theta_1)=&\sum_{C\in\mathcal{F}_n}e^{-H(C)}=\sum_{C\in\mathcal{F}_n}e^{-\theta_1f_1(C)}=\sum_{\a_2}e^{-\theta_1\sum\limits_{\i_2}a_{\i_2}}\\
=&\sum_{\a_2}\prod_{\i_2}e^{-\theta_1a_{\i_2}}=\prod_{\i_2}\sum_{a_{\i_2}=0}^1e^{-\theta_1a_{\i_2}}=(1+e^{\theta_1})^{{n\choose2}}.
\end{split}
\end{equation}
We can now  solve (\ref{eq:parametersC}) with $\bar{x}_1={n\choose2}p$ for the parameter $\theta_1$,
\begin{equation}\label{eq:theta1_X}
\theta_1=-\ln\frac{p}{1-p},
\end{equation}
and check that indeed
\begin{equation}
\mathbb{P}_{X(n,p)}(C)=\frac{e^{-\theta_1f_1(C)}}{Z(\theta_1)}.
\end{equation}
This completes the proof.
\end{proof}
Given (\ref{eq:G=ERSC}), the result in (\ref{eq:X=ERSC}) is intuitively expected since the random part of generating $C\sim X(n,p)$ is sampling the $1$-skeleton $C^{(1)}\sim G(n,p)$. The rest of the construction, $C=X(C^{(1)})$, is fully deterministic.

\subsection{Linial--Meshulam Random Complexes}

Another example of ERSC is a generative model $Y(n,p)$ for random $2$-complexes. To generate $Y\sim Y(n,p)$, we start with a complete graph on $n$ vertices, the $1$-skeleton of a future simplicial complex, and add each of the ${n\choose 3}$ possible triangle faces independently at random with probability $p$. Linial and Meshulam introduced this model in \cite{LiMe06} and studied its topological properties. In particular, they proved for $Y(n,p)$ a cohomological analog of the celebrated Erd\H{o}s--R\'{e}nyi theorem on connectivity of the Erd\H{o}s--R\'{e}nyi random graphs \cite{ErRe59}. The model $Y(n,p)$ can be readily generalized to higher dimensions: start with a full $d$-complex on $n$ vertices, $1\leq d\leq n-2$, and add each of the ${n \choose d+2}$ possible $(d+1)$-simplices independently at random with probability $p$. We denote this model by $Y_d(n,p)$. The original Linial--Meshulam random complex $Y(n,p)$ is then  $Y_1(n,p)$.

Let $\mathcal{C}_n^{(d+1)}\subset\mathcal{C}_n$ be a set of all simplicial complexes of dimension $(d+1)$ or less, and  $\mathcal{Y}_d\subset\mathcal{C}_n^{(d+1)}$ be a subset of  complexes with full $d$-skeleton. In other words,
\begin{equation}
\mathcal{Y}_d=\{C\in\mathcal{C}_n^{(d+1)}: C^{(k)}=C^{[k]}, k=1,\ldots,d\}.
\end{equation}
Since for any $C\in\mathcal{Y}_d$, the first $(d+1)$ adjacency tensors $\a_1,\ldots,\a_{d+1}$ are unit tensors with zero diagonals,  $\mathcal{Y}_d=\a_{d+2}$.
\begin{proposition}
	The Linial--Meshulam random complex $Y_d(n,p)$ is the ERSC ensemble:
	\begin{equation}\label{eq:Y=ERSC}
	Y_d(n,p)=\mathrm{ERSC}\left(\mathcal{Y}_d,f_{d+1},{n \choose d+2}p\right).
	\end{equation}
\end{proposition}
 \begin{proof}
The proof is similar to that for random flag complexes.  Given $C\in\mathcal{Y}_d$, the probability that the complex has been generated by $Y_d(n,p)$ is
\begin{equation}
\label{eq:P_Y}
\mathbb{P}_{Y_d(n,p)}(C)=p^{f_{d+1}(C)}(1-p)^{{n\choose d+2}-f_{d+1}(C)}.
\end{equation}
We need to show that this is in fact the maximum-entropy distribution under the constraint $\mathbb{E}[f_{d+1}]={n \choose d+2}p$. The partition function:
\begin{equation}\label{eq:Z for Y}
\begin{split}
Z(\theta_1)=&\sum_{C\in\mathcal{Y}_d}e^{-H(C)}=\sum_{C\in\mathcal{Y}_d}e^{-\theta_1f_{d+1}(C)}=\sum_{\a_{d+2}}e^{-\theta_1\sum\limits_{\i_{d+2}}a_{\i_{d+2}}}\\
=&\sum_{\a_{d+2}}\prod_{\i_{d+2}}e^{-\theta_1a_{\i_{d+2}}}=\prod_{\i_{d+2}}\sum_{a_{\i_{d+2}}=0}^1e^{-\theta_1a_{\i_{d+2}}}=(1+e^{\theta_1})^{{n\choose d+2}}.
\end{split}
\end{equation}
The Lagrange multiplier is then $\theta_1=-\ln\frac{p}{1-p} $, and
\begin{equation}
\mathbb{P}_{Y_d(n,p)}(C)=\frac{e^{-\theta_1f_{d+1}(C)}}{Z(\theta_1)},\end{equation}
as claimed.
\end{proof}

This result is also expected, since the Linial--Meshulam random complex is a higher dimensional analog of the Erd\H{o}s--R\'{e}nyi random graph: sampling from $Y_d(n,p)$ is the same Bernoulli trials process as in $G(n,p)$, with the only difference being that now we are creating $(d+1)$-simplices instead of 1-simplices (edges).

\section{Any Distribution is Maximum-Entropy}\label{sec:maximum_entropy}

	
It is a well-known fact in statistics and information theory (\eg \cite{CoTho91}) that \textit{any} discrete distribution $\mathbb{P}^*$ is maximum-entropy under properly specified constraints. Specifically, if one can write $-\ln\mathbb{P}^*$ as a linear combination $\sum\lambda_ih_i^*+\xi$ of some functions $\{h_i^*\}$, then distribution $\mathbb{P}^*$ uniquely maximizes entropy $S(\mathbb{P})$ across all distributions $\mathbb{P}$ that satisfy constraints $\mathbb{E}_{\mathbb{P}}[h_i^*]=\mathbb{E}_{\mathbb{P}^*}[h_i^*]$. In this section, we briefly review this general result, and show how it applies to the already considered models $G(n,p)$, $X(n,p)$, and $Y_d(n,p)$, where $-\ln\mathbb{P}^*$ can be written as a linear combination. We will see in the next section that this result simplifies dramatically the proofs for more complicated ERSCs.

Let us consider a discrete probability space $(\Omega,\mathbb{P}^*)$, where $\Omega$ is a finite sample space and $\mathbb{P}^*$ is some fixed probability distribution on $\Omega$. Let us represent the distribution $\mathbb{P}^*$ in the ``Gibbs form'' as follows:
\begin{equation}
\mathbb{P}^*(\omega)=e^{-(-\ln\mathbb{P}^*(\omega))}=e^{-H^*(\omega)},
\end{equation}
where
\begin{equation}\label{eq:Hamiltonian}
H^*(\omega)=-\ln\mathbb{P}^*(\omega).
\end{equation}
Let $\bar{H}^*$ denote the expectation of the function  $H^*:\Omega\rightarrow\mathbb{R}$ with respect to $\mathbb{P}^*$, which is exactly the entropy of $\mathbb{P}^*$,
\begin{equation}\label{eq:entropy*}
\bar{H}^*=\mathbb{E}_{\mathbb{P}^*}[H^*]=\sum_{\omega\in\Omega}H^*(\omega)\mathbb{P}^*(\omega)=S(\mathbb{P}^*).
\end{equation}
\begin{lemma}\label{lemma1} The probability distribution $\mathbb{P}^*$ is the solution of the following
	optimization problem:
	\begin{equation}\label{eq:opt_prob_lemma1}
	S(\mathbb{P})=-\sum_{\omega\in\Omega}\mathbb{P}(\omega)\ln\mathbb{P}(\omega)\rightarrow \max,
	\end{equation}
	subject to the  constraints
	\begin{equation}\label{eq:opt_prob_lemma1_constr}
	\sum_{\omega\in\Omega}\mathbb{P}(\omega)=1 \hspace{3mm}\mbox{and} \hspace{3mm}
	\mathbb{E}_\mathbb{P}[H^*]\equiv-\sum_{\omega\in\Omega}\mathbb{P}(\omega)\ln\mathbb{P}^*(\omega)=\bar{H}^*.
	\end{equation}
\end{lemma}

In other words, Lemma~\ref{lemma1}  states that any discrete probability distribution is a maximum-entropy distribution. The entropy maximization is across all possible distributions $\mathbb{P}$ satisfying the constraint that the expected value of $H^*$ in distribution $\mathbb{P}$ is equal to $H^*$'s expected value in distribution $\mathbb{P}^*$, which is $\mathbb{P}^*$'s entropy. In what follows, we will need a more general version of Lemma~\ref{lemma1}.

Suppose that function $H^*$ in (\ref{eq:Hamiltonian}) can be written as a linear combination of $r$ other functions $h_i^*:\Omega\rightarrow\mathbb{R}$:
\begin{equation}\label{eq:linearHam}
H^*(\omega)=-\ln\mathbb{P}^*(\omega)=\sum_{i=1}^r \lambda_i h_i^*(\omega) + \xi,
\end{equation}
where $\lambda_i,\xi\in\mathbb{R}$ are constants.
Let $\bar{h}^*_i$ denote the expectation of $h^*_i$ with respect to $\mathbb{P}^*$,
\begin{equation}
\bar{h}^*_i=\mathbb{E}_{\mathbb{P}^*}[h^*_i]=\sum_{\omega\in\Omega}h_i^*(\omega)\mathbb{P}^*(\omega).
\end{equation}
\begin{lemma}(Th.\ 11.1.1~\cite{CoTho91})\label{lemma2}
	The probability distribution $\mathbb{P}^*$ is the solution of the following optimization problem:
	\begin{equation}\label{eq:opt_prob_lemma2}
	S(\mathbb{P})=-\sum_{\omega\in\Omega}\mathbb{P}(\omega)\ln\mathbb{P}(\omega)\rightarrow \max,
	\end{equation}
	subject to the constraints
	\begin{equation}\label{eq:opt_prob_lemma2_constr}
	\sum_{\omega\in\Omega}\mathbb{P}(\omega)=1 \hspace{3mm}\mbox{and}\hspace{3mm}
	\mathbb{E}_\mathbb{P}[h^*_i]\equiv\sum_{\omega\in\Omega}h_i^*(\omega)\mathbb{P}(\omega)=\bar{h}^*_i, \hspace{3mm}
	i=1,\ldots,r.
	\end{equation}
\end{lemma}
We note that the main utility of Lemma~\ref{lemma2} is not in observing that $\mathbb{P}^*(\omega)\propto e^{-\sum_{i=1}^r \lambda_i h_i^*(\omega)}$ is a maximum-entropy distribution (in fact, Lemma~\ref{lemma1} states that any distribution is), but in specifying more general \textit{constraints} (\ref{eq:opt_prob_lemma2_constr}) under which distribution $\mathbb{P}^*$ is maximum-entropy. Lemma~\ref{lemma1} is a special case of more general Lemma~\ref{lemma2} with $\xi=0, r=1$, and $\lambda_1=1$. Indeed, in this case $H^*=h_1^*$, and the constraints in (\ref{eq:opt_prob_lemma1_constr}) and (\ref{eq:opt_prob_lemma2_constr}) become manifestly identical. Lemma~\ref{lemma2} is identical to Theorem~11.1.1 in~\cite{CoTho91}, but we provide the proof here for completeness.

\begin{proof}
	Let $\mathbb{P}$ be any distribution that satisfies the constraints in (\ref{eq:opt_prob_lemma2_constr}). Then its entropy
	\begin{equation}
	\begin{split}
	 S(\mathbb{P})=&-\sum_{\omega\in\Omega}\mathbb{P}(\omega)\ln\mathbb{P}(\omega)=-\sum_{\omega\in\Omega}\mathbb{P}(\omega)\ln\frac{\mathbb{P}(\omega)\mathbb{P}^*(\omega)}{\mathbb{P}^*(\omega)}\\
	=&-D_{\mathrm{KL}}(\mathbb{P}\hspace{-0.5mm}\parallel\hspace{-0.5mm}\mathbb{P}^*)+\sum_{\omega\in\Omega}\mathbb{P}(\omega)H^*(\omega),
	\end{split}
	\end{equation}
	where $D_{\mathrm{KL}}(\mathbb{P}\hspace{-0.5mm}\parallel\hspace{-0.5mm}\mathbb{P}^*)$ is the Kullback--Leibler (KL) divergence of $\mathbb{P}$ from $\mathbb{P}^*$. Since the KL divergence is always non-negative,
	\begin{equation}
	\begin{split}
	S(\mathbb{P})\leq& \sum_{\omega\in\Omega}\mathbb{P}(\omega)H^*(\omega)=\sum_{\omega\in\Omega}\mathbb{P}(\omega)\left(\sum_{i=1}^r \lambda_i h_i^*(\omega) + \xi\right)\\
	=&\sum_{i=1}^r\lambda_i\mathbb{E}_{\mathbb{P}}[h_i^*]+\xi=
	\sum_{i=1}^r\lambda_i\bar{h}_i^*+\xi=S(\mathbb{P}^*).
	\end{split}
	\end{equation}
	This shows that $\mathbb{P}^*$ indeed maximizes the entropy. The uniqueness follows form the fact that $D_{\mathrm{KL}}(\mathbb{P}\hspace{-0.5mm}\parallel\hspace{-0.5mm}\mathbb{P}^*)=0$ if and only if $\mathbb{P}=\mathbb{P}^*$.
\end{proof}

Lemmas~\ref{lemma1}\&\ref{lemma2} can be formulated for any  ensemble of discrete ``objects,'' including sets of graphs and simplicial complexes. Using the notation introduced in Definition~\ref{def2} and applying the Lemmas to $\Omega=\mathcal{S}\subset\mathcal{C}_n$, we can concisely write Lemma~\ref{lemma1} as
\begin{equation}\label{eq:lemma1}
(\mathcal{S},\mathbb{P}^*)=\mathrm{ERSC}\left(\mathcal{S},H^*,\bar{H}^*\right),
\end{equation}
and Lemma~\ref{lemma2} as
\begin{equation}\label{eq:lemma2}
(\mathcal{S},\mathbb{P}^*)=\mathrm{ERSC}\left(\mathcal{S},\{h_i^*\},\{\bar{h}_i^*\}\right).
\end{equation}

In many cases, Lemmas~\ref{lemma1}\&\ref{lemma2} are not useful, since for many generative models $(\mathcal{S},\mathbb{P}^*)$ the distribution $\mathbb{P}^*$  cannot be explicitly written in the Gibbs form with linear Hamiltonian (\ref{eq:linearHam}). 
 Moreover, in generative models, the complexity of the generating algorithm often makes it impossible to  even explicitly  compute the probability $\mathbb{P}^*(C)$ for a given $C$. Even in the preferential attachment model \cite{BarAlb99}, where the algorithm which generates a network appears to be fairly simple --- a new node connects to existing node $i$ with probability proportional to its degree $p_i\propto k_i$ --- the resulting distribution is unknown.
However, if we do know $\mathbb{P}^*$ as a function of observables, $\mathbb{P}^*(C)\propto e^{-\sum_{i=1}^r \lambda_i h_i^*(C)}$, Lemma~\ref{lemma2} is very helpful in representing $(\mathcal{S},\mathbb{P}^*)$ as an ERSC.

Indeed, let us briefly see how Lemma~\ref{lemma2} applies to the already considered generative models. For $G(n,p)$, the probability of a graph $G\in\mathcal{G}_n$ in the model is
\begin{equation}
\mathbb{P}_{G(n,p)}=p^{f_1(G)}(1-p)^{{n\choose 2}-f_1(G)}.
\end{equation}
The corresponding Hamiltonian is then
\begin{equation}
\begin{split}
H_{G(n,p)}(G)=&-f_1(G)\ln p-\left({n\choose 2}-f_1(G)\right)\ln(1-p)\\
=&\underbrace{\ln\frac{1-p}{p}}_{\lambda_1}\underbrace{f_1(G)}_{h_1^*(G)}-\underbrace{{n\choose 2}\ln(1-p)}_{\xi},
\end{split}
\end{equation}
where the bottom notations refer to the notations in Lemma~\ref{lemma2}. The observation that $G(n,p)$ is an ERG (\ref{eq:G=ERSC}) then follows from Lemma~\ref{lemma2}, since $\mathbb{E}_{\mathbb{P}_{G(n,p)}}[f_1]={n\choose 2}p$. Similarly, for $X(n,p)$ and $Y_d(n,p)$,
\begin{equation}
\begin{split}
&H_{X(n,p)}(C)=\ln\frac{1-p}{p}f_1(C)-{n\choose 2}\ln(1-p), \\ &\bar{f}_{1}=\mathbb{E}_{{X(n,p)}}[f_1]={n\choose 2}p, \\
&H_{Y_d(n,p)}(C)=\ln\frac{1-p}{p}f_1(C)-{n\choose d+1}\ln(1-p), \\ &\bar{f}_{d+1}=\mathbb{E}_{{Y_d(n,p)}}[f_{d+1}]={n\choose d+2}p,
\end{split}
\end{equation}
and the observations (\ref{eq:X=ERSC}) and (\ref{eq:Y=ERSC}) that these ensembles are ERSCs are direct corollaries of Lemma~\ref{lemma2}.

The main point of this section is that in case the probability distribution is a known exponential function of a linear combination of structural observables, the computation of the partition function, which tends to be a nontrivial task in general, is not necessary  to show that the distribution is the unique maximizer of entropy across all the distributions that satisfy the constraints that the expected values of these observables are equal to their expected values in this distribution.

\section{Kahle's $\Delta$-Ensembles}\label{sec:Khale}
We now turn to a more general model which contains the Erd\H{o}s--R\'{e}nyi random graphs, the random flag complexes, and the Linial--Meshulam complexes as special cases. In a recent survey~\cite{Kahle14}, Kahle  introduced the following multi-parameter model $\Delta(n;p_1,\ldots,p_{n-1})$ which generates random simplicial complexes inductively by dimension. First, build a $1$-skeleton by putting an edge between any two vertices with probability $p_1$. Then, for $d=2,\ldots,n-1$, add every $d$-simplex with probability $p_d$, but only if the entire $(d-1)$-dimensional boundary of that simplex is already in place. More formally,  we have the following definition.
\begin{definition}
The Kahle model $\Delta(n;p_1,\ldots,p_{n-1})$ is a random simplicial complex model that generates $C\in\mathcal{C}_{ n}$ as follows: for $d=1,\ldots,n-1$, for every $\i_{d+1}$,
\begin{equation}
\begin{split}
\mbox{if } b_{\i_{d+1}}=0  \hspace{2mm}\Rightarrow \hspace{2mm}& \mbox{set } a_{\i_{d+1}}=0.\\
\mbox{if } b_{\i_{d+1}}=1  \hspace{2mm}\Rightarrow \hspace{2mm}& \mbox{set } a_{\i_{d+1}}=\begin{cases} 1 &\mbox{with probability } p_d,  \\
0 & \mbox{with probability } 1-p_d. \end{cases}
\end{split}
\end{equation}
\end{definition}
Topological properties of the Khale model are studies in \cite{CoFa1,CoFa2,CoFa3}. Here we study its maximum-entropy properties. In Appendix~\ref{sec:A1} we prove the following proposition.
\begin{proposition}\label{prop:3}
	Let $C\sim\Delta(n;p_1,\ldots,p_{n-1})$. The expected numbers of $d$-simplices in $C^{(d)}$ and $C^{[d]}$ are
	\begin{equation}
	\bar{f}_d={n\choose d+1}\prod_{k=1}^dp_k^{d+1\choose d-k} \hspace{3mm} \mbox{ and } \hspace{3mm} \bar{\phi}_d={n\choose d+1}\prod_{k=1}^{d-1}p_k^{d+1\choose d-k}.
	\end{equation}
\end{proposition}

The Kahle model unifies all the random simplicial complexes we have considered so far:
\begin{equation}
\begin{split}
G(n,p)&=\Delta(n;p,0,\ldots,0),\\
X(n,p)&=\Delta(n;p,1,\ldots,1),\\
Y(n,p)&=\Delta(n;1,p,0,\ldots,0),\\
Y_d(n,p)&=\Delta(n;\underbrace{1,\ldots,1}_{d},p,0,\ldots,0).
\end{split}
\end{equation}
Since all these special cases are ERSCs, it is natural to expect that so is $\Delta(n;p_1,\ldots,p_{n-1})$. We cannot prove this using the same method as for the Erd\H{o}s--R\'{e}nyi random graphs and the random flag and Linial--Meshulam complexes in Section~\ref{sec:simple examples}. As with ERGs, analytical computation of the partition function $Z(\theta)$ for ERSCs is rarely possible, and $G(n,p)$, $X(n,p)$, and $Y_d(n,p)$ are lucky exceptions. In Appendix~\ref{sec:A3} we illustrate difficulties one has to be prepared to experience when attempting to compute the partition function for $\Delta(n;p_1,\ldots,p_{n-1})$ with $n=3$.
However, with the help of Lemmas~\ref{lemma1}\&\ref{lemma2} in Section~\ref{sec:maximum_entropy} there exists a simpler alternative proof. The fact that $\Delta(n;p_1,\ldots,p_{n-1})$ is an ERSC is a direct corollary of those lemmas.
\begin{theorem}\label{thm:Kahle}
	The Kahle  $\Delta$-ensemble is the ERSC ensemble:
	\begin{equation}
	\label{eq:Delta=ERSC}
	\Delta(n;p_1,\ldots,p_{n-1})= 
	\mathrm{ERSC}\left(\mathcal{C}_n,\left\{\{f_d\}_{d=1}^{n-1},\{\phi_d\}_{d=2}^{n-1}\right\},\left\{\{\bar{f}_d\}_{d=1}^{n-1},\{\bar{\phi}_d\}_{d=2}^{n-1}\right\}\right),
	\end{equation}
	where $\bar{f}_d$ and $\bar{\phi}_d$ are the expected  numbers of $d$-simplices in $C^{(d)}$ and $C^{[d]}$.
\end{theorem}
\begin{proof}
For any $C\in\mathcal{C}_n$, the probability $\mathbb{P}_\Delta(C)$ that $\Delta(n,p_1,\ldots,p_{n-1})$ generates $C$ can be computed by induction:
\begin{equation}\label{eq:Delta}
\begin{split}
\mathbb{P}_\Delta(C)=& \prod_{d=1}^{n-1}\mathbb{P}_\Delta\left.\left(C^{(d)}\right|C^{(d-1)}\right)=\prod_{d=1}^{n-1}p_d^{f_d(C)}(1-p_d)^{\phi_d(C)-f_d(C)}.
\end{split}
\end{equation}
Indeed, given the $(d-1)$-skeleton $C^{(d-1)}$, the maximum possible number of $d$-simplices in $C^{(d)}$ is exactly $\phi_d(C)$, the number of $d$-simplices in the filled skeleton $C^{[d]}$. Since each of these $d$-simplices appears independently with probability $p_d$, the conditional probability $\mathbb{P}_\Delta\left.\left(C^{(d)}\right|C^{(d-1)}\right)=p_d^{f_d(C)}(1-p_d)^{\phi_d(C)-f_d(C)}$, where $f_d(C)$ is the actual number of $d$-simplices in $C^{(d)}$.

The Hamiltonian of $C$ is therefore
\begin{equation}
\begin{split}
H_\Delta(C)=&\sum_{d=1}^{n-1}\left(f_d(C)\ln\frac{1-p_d}{p_d}+\phi_d(C)\ln\frac{1}{1-p_d}\right)\\
=&\sum_{d=1}^{n-1}f_d(C)\ln\frac{1-p_d}{p_d} + \sum_{d=2}^{n-1}\left(\phi_d(C)\ln\frac{1}{1-p_d}\right) + {n\choose 2}\ln\frac{1}{1-p_{1}},
\end{split}
\end{equation}
since $\phi_1(C)={n\choose 2}$ for any $C$. Using Lemma~\ref{lemma2} with
\begin{equation}
\begin{split}
\{h_i^*\}=&\{f_1,\ldots,f_{n-1},\phi_2,\ldots,\phi_{n-1}\},\\
\{\lambda_i\}=&\left\{\ln\frac{1-p_1}{p_1},\ldots,\ln\frac{1-p_{n-1}}{p_{n-1}},\ln\frac{1}{1-p_2},\ldots,\ln\frac{1}{1-p_{n-1}}\right\},\\
\xi=&{n\choose 2}\ln\frac{1}{1-p_{1}},\\
\end{split}
\end{equation}
completes the proof.
\end{proof}

\section{General Random Simplicial Complexes with Independent Simplices}

Finally, we introduce and consider the most general case of random simplicial complexes with statistically independent simplices. In this case each simplex has its own individual probability of appearance. To stay as general as possible, we must allow for even the $0$-simplices (vertices) to be present with any probabilities, which are not necessarily equal to $1$. We denote this new model by $\D(n;\p_1,\ldots,\p_n)$, or $\D$ for brevity, where $\p_d=\{p_{\i_d}\}$ is a collection of ${n\choose d}$ appearance probabilities for each  $(d-1)$-simplex. Whereas in $\Delta$ in the previous section, the subindex $d$ in $p_d$ refers to the simplex dimension, in $\D$ the sub-multi-index $\i_d$ in $p_{\i_d}$ refers to the specific $(d-1)$-simplex $\{\i_d\}$. To generate $C\sim\D$, we first create its $0$-skeleton by having vertices $\{\i_1\}\in C$ with probabilities $p_{\i_1}$, $\i_1=1,\ldots,n$. Then, for $d=1,\ldots,n-1$, we add every $d$-simplex $\{\i_{d+1}\}$ with probability $p_{\i_{d+1}}$, but only if the entire $(d-1)$-dimensional boundary of that simplex is already in place. Figure~\ref{fig4} illustrates the generation of a $2$-complex from $\D$ with $n=10$.

\begin{figure}
	\centerline{\includegraphics[width=140mm]{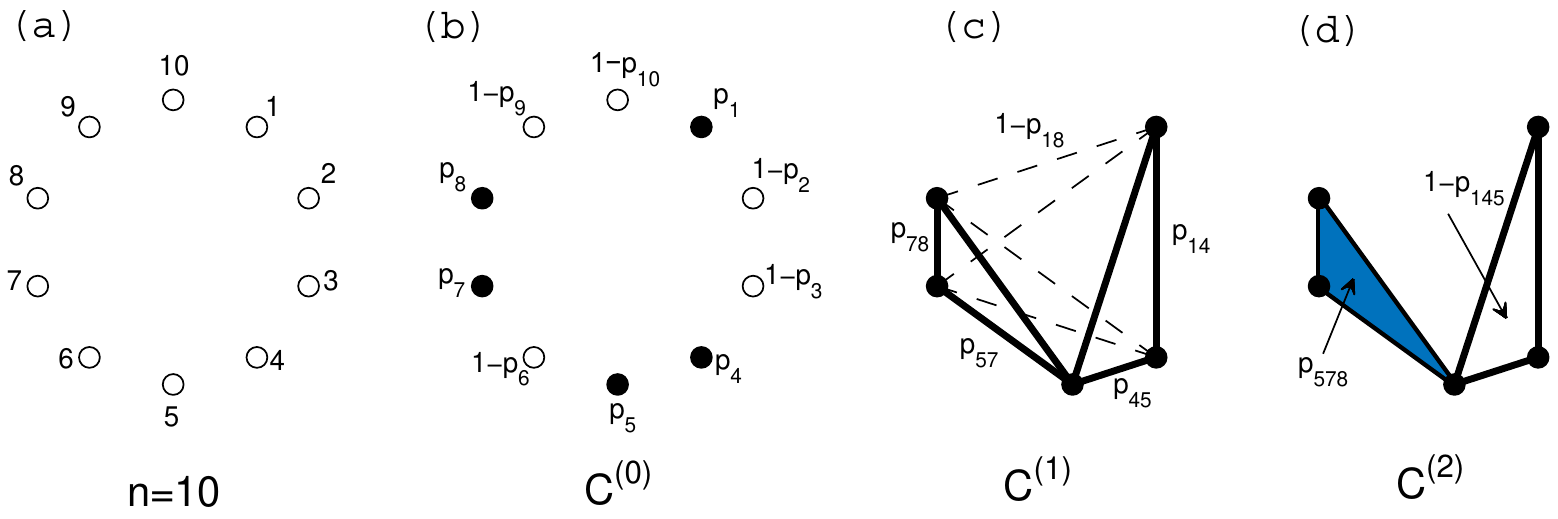}}
	\caption{\small \textbf{Sampling from $\D$ with $n=10$.} Panel~(a) shows the starting point: $n=10$ potential $0$-simplices (vertices) represented by empty dots. At the first step, shown in Panel~(b), we create the $0$-skeleton of a future complex $C$ by including (excluding) every vertex $i$ in (from) $C^{(0)}$ with probability $p_i$ (with probability $1-p_i$). In this example, five vertices represented by filled dots belong to  $C^{(0)}$. Next, in Panel~(c), we generate the $1$-skeleton by creating $1$-simplices (edges). Each of the ${5 \choose 2}$ possible edges $\{\i_2\}=\{1,4\},\{1,5\},\ldots,\{7,8\}$ appears in $C^{(1)}$ with probability $p_{\i_2}$. We represent the accepted (rejected) edges by solid (dashed) lines. Finally, in Panel~(d), we create the $2$-skeleton by adding $2$-simplices (triangles). There are only two possible triangles in $C^{(2)}$: $\{1,4,5\}$ and $\{5,7,8\}$. Here, the former (empty) was rejected with probability $1-p_{145}$, and the latter (filled) was accepted with probability $p_{578}$.  }
	\label{fig4}
\end{figure}

More formally, we have the following definition. Let $\mathcal{C}_{\leq n}=\cup_{k=1}^n\mathcal{C}_k$
denote the set of all simplicial complexes with $n$ vertices or less. As in Section~\ref{notation}, any $C\in\mathcal{C}_{\leq n}$ is uniquely determined by a collection of its adjacency tensors $\a_d=\{a_{\i_d}\}$, $d=1,\ldots,n$, except that now $\a_1$ is not necessarily equal to the all-ones vector  ${\bf 1}_n$, since $C$ may have less than $n$ vertices.
\begin{definition} The model \label{def:fat-delta}
$\D(n;\p_1,\ldots,\p_n)$ is a random simplicial complex model that generates $C\in\mathcal{C}_{\leq n}$ as follows:
for $d=0,\ldots,n-1$, for every $\i_{d+1}$,
	\begin{equation}
	\begin{split}
	\mbox{if } b_{\i_{d+1}}=0  \hspace{2mm}\Rightarrow \hspace{2mm}& \mbox{set } a_{\i_{d+1}}=0,\\
	\mbox{if } b_{\i_{d+1}}=1  \hspace{2mm}\Rightarrow \hspace{2mm}& \mbox{set } a_{\i_{d+1}}=\begin{cases} 1 &\mbox{with probability } p_{\i_{d+1}},  \\
	0 & \mbox{with probability } 1-p_{\i_{d+1}}. \end{cases}
	\end{split}
	\end{equation}
\end{definition}
Here, for convenience, we use the convention that $b_{\i_1}=1$ for all $\i_1=1,\ldots,n$.
If $p_{\i_{d+1}}=p_d$ for $d=0,\ldots,n-1$ and $p_0=1$, then we recover the original Kahle's model $\Delta(n;p_1,\ldots,p_{n-1})$ from the previous section.

To give the expressions for the expected values of the observables ${a}_{\i_d}$ and ${b}_{\i_d}$ in $\D(n;\p_1,\ldots,\p_n)$, we need a bit of new notation. Let multi-index $\k_m=k_1,\ldots,k_m$ denote an $m$-tuple with increasing values $1\leq k_1<\ldots<k_m \leq d$, and $\i_d^{\widehat{\k}_m}=i_1,\ldots,\widehat{i_{k_1}},\ldots,\widehat{i_{k_m}},\ldots,i_d$ be the $(d-m)$-long multi-index with $i_{k_1},\ldots,i_{k_m}$ omitted, with a convention that $\k_0=\varnothing$ and $\i_d^{\widehat{\k}_0}=\i_d$. In Appendix~\ref{sec:A2} we prove the following proposition.
\begin{proposition}\label{prop:4}
	Let $C\sim\D(n;\p_1,\ldots,\p_{n})$. The expected values of the observables ${a}_{\i_d}$ and ${b}_{\i_d}$ are
	\begin{equation}\label{eq:exp_a_b}
	\bar{a}_{\i_d}=\prod_{m=0}^{d-1}\prod_{\k_m}p_{\i_d^{\widehat{\k}_m}} \hspace{3mm} \mbox{and} \hspace{3mm} \bar{b}_{\i_d}=\prod_{m=1}^{d-1}\prod_{\k_m}p_{\i_d^{\widehat{\k}_m}}.
	\end{equation}
\end{proposition}

Lemma~\ref{lemma2} helps again to prove that the general model $\D$ is also an ERSC.
\begin{theorem}\label{thm:3}
	The $\D$-ensemble is the ERSC ensemble:
	\begin{equation}
	\label{eq:D_gen=ERSC}
	\D(n;\p_1,\ldots,\p_{n})=\mathrm{ERSC}\left(\mathcal{C}_{\leq n},\left\{\{a_{\i_d}\}_{d=1}^n,\{b_{\i_d}\}_{d=2}^n\right\},\left\{\{\bar{a}_{\i_d}\}_{d=1}^n,\{\bar{b}_{\i_d}\}_{d=2}^n\right\}\right),
	\end{equation}
	where $\bar{a}_{\i_d}$ and $\bar{b}_{\i_d}$ are the expected values of the observables $a_{\i_d}$ and $b_{\i_d}$.
\end{theorem}
\begin{proof}
The probability that $\D(n;\p_1,\ldots,\p_{n})$ generates  $C\in\mathcal{C}_{\leq n}$ is
\begin{equation}\label{eq:Delta_general}
\begin{split}
\mathbb{P}_\D(C)=&\mathbb{P}_\D\left(C^{(0)}\right) \prod_{d=1}^{n-1}\mathbb{P}_\D\left(C^{(d)}|C^{(d-1)}\right)\\=&
\prod_{\i_1}p_{\i_1}^{a_{\i_1}}(1-p_{\i_1})^{1-a_{\i_1}}\times\prod_{d=1}^{n-1}\prod_{\i_{d+1}}p_{\i_{d+1}}^{a_{\i_{d+1}}}(1-p_{\i_{d+1}})^{b_{\i_{d+1}}-a_{\i_{d+1}}}\\
=&\prod_{d=1}^{n}\prod_{\i_{d}}p_{\i_{d}}^{a_{\i_{d}}}(1-p_{\i_{d}})^{b_{\i_{d}}-a_{\i_{d}}}.
\end{split}
\end{equation}
The Hamiltonian of $C$ is then
\begin{equation}
\begin{split}
H_\D(C)=&\sum_{d=1}^{n}\sum_{\i_{d}}\left(a_{\i_d}\ln\frac{1-p_{\i_d}}{p_{\i_d}}+b_{\i_d}\ln\frac{1}{1-p_{\i_d}}\right)\\=&\sum_{d=1}^{n}\sum_{\i_{d}}\alpha_{\i_d}a_{\i_d}+\sum_{d=2}^{n}\sum_{\i_{d}}\beta_{\i_d}b_{\i_d}+\sum_{\i_1}\ln\frac{1}{1-p_{\i_1}},
\end{split}
\end{equation}
where $\alpha_{\i_d}$ and $\beta_{\i_d}$ are the Lagrange multipliers coupled to observables $a_{\i_d}$ and $b_{\i_d}$,
\begin{equation}
\alpha_{\i_d}=\ln\frac{1-p_{\i_d}}{p_{\i_d}} \hspace{3mm} \mbox{and} \hspace{3mm} \beta_{\i_d}=\ln\frac{1}{1-p_{\i_d}}.
\end{equation}
Using Lemma~\ref{lemma2} with
\begin{equation}
\begin{split}
\{h_i^*\}=&\left\{\{a_{\i_1}\},\ldots,\{a_{\i_n}\},\{b_{\i_2}\},\ldots,\{b_{\i_n}\}\right\},\\
\{\lambda_i\}=&\left\{\{\alpha_{\i_1}\},\ldots,\{\alpha_{\i_n}\},\{\beta_{\i_2}\},\ldots,\{\beta_{\i_n}\}\right\},\\
\xi=&\sum_{\i_1}\ln\frac{1}{1-p_{\i_1}},
\end{split}
\end{equation}
completes the proof.
\end{proof}

\section{Discussion}

In summary, exponential random simplicial complexes (ERSCs) are a natural higher dimensional analog of exponential random graphs which are extensively used for modeling network data and statistical inference. An ERSC ensemble is a maximum-entropy ensemble of simplicial complexes under ``soft'' constraints that fix expected values of some observables or properties of simplicial complexes.
We have developed the formalism for ERSCs, and introduced
the most general generative model of random simplicial complexes $\D$ with statistically independent simplices. This model has as special cases several popular models studied in the literature: Erd\H{o}s--R\'{e}nyi random graphs, random flag complexes, Linial--Meshulam complexes, and Kahle's $\Delta$-ensembles. As all these models, $\D$ is an ERSC ensemble. The constraints in this ensemble are expected number of simplices and their boundaries.

This result is a direct corollary of the general observation that any probability distribution $\mathbb{P}$ is maximum-entropy under the constraint that the expected value of $-\ln\mathbb{P}$ is equal to the entropy of $\mathbb{P}$.
This observation dramatically simplifies the representation of many ensembles of random simplicial complexes as ERSCs since the calculation of the partition function is no longer needed. For example, to show that the Erd\H{o}s--R\'{e}nyi random graphs $G(n,p)$ are exponential random graphs with a given expected number of edges, one does not really have to calculate the partition function. This calculation is trivial in the Erd\H{o}s--R\'{e}nyi case or in the general case of exponential random graphs with statistically independent edges \cite{PaNe04}. However, the analogous calculation for the general case of random simplicial complexes $\D$ with statistically independent simplices appears to be intractable.

The multi-parameter model $\D(n;\p_1,\ldots,\p_n)$ is the ERSC ensemble with two types of constrained observables: $\{a_{\i_d}\}$ and $\{b_{\i_d}\}$. The observables of the first type are simplices themselves: $a_{\i_d}(C)=1$ if the $(d-1)$-simplex $\{i_d\}$ belongs to $C$, and zero otherwise. The observables of the second type are their boundaries: $b_{\i_d}(C)=1$ if the entire   $(d-2)$-dimensional boundary of simplex $\{\i_d\}$ belongs to $C$, and zero otherwise. Theorem~\ref{thm:3} states that $\D(n;\p_1,\ldots,\p_n)$ is a solution of the following optimization problem:
\begin{align}
&S(\mathbb{P})\rightarrow\max, \hspace{5mm} \sum_{C\in\mathcal{C}_{\leq n}}\mathbb{P}(C)=1,\\  &\mathbb{E}_{\mathbb{P}}[a_{\i_d}]=\bar{a}_{\i_d}, d=1,\ldots,n, \hspace{5mm}  \mathbb{E}_{\mathbb{P}}[b_{\i_d}]=\bar{b}_{\i_d}, d=2,\ldots,n. \label{eq:soft}
\end{align}
If we drop the observables of the second type in this optimization problem, we alter the maximum-entropy distribution as illustrated in Figure~\ref{fig5}. Since the distribution has changed, ensemble $\L(n;\p_1,\ldots,\p_n)=\mathrm{ERSC}\left(\mathcal{C}_{\leq n},\{a_{\i_d}\}_{d=1}^n,\{\bar{a}_{\i_d}\}_{d=1}^n\right)$ defined by this distribution is now also different from $\D(n;\p_1,\ldots,\p_n)$.

\begin{figure}
	\centerline{\includegraphics[width=70mm]{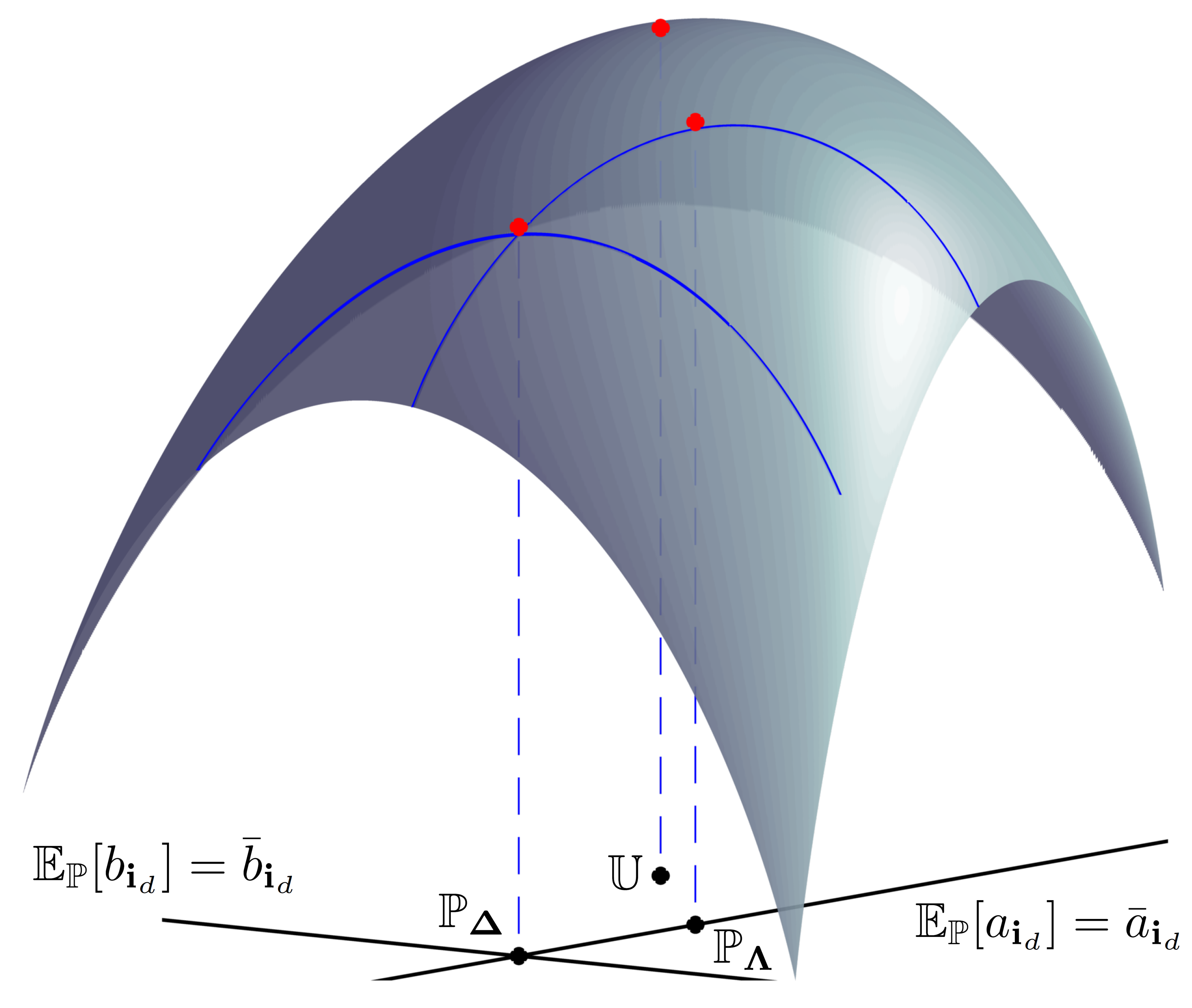}}
	\caption{\small \textbf{Constrained entropy maximization.} The surface represents the Gibbs entropy $S$, which, in this schematic example, is a function on the set of all probability distributions on $\mathcal{C}_{\leq n}$. The global maximum corresponds to the uniform distribution $\mathbb{U}$, which is the maximum-entropy distribution among all distributions supported on $\mathcal{C}_{\leq n}$. Theorem~\ref{thm:3} shows that if we have two sets of constraints, $\mathbb{E}_{\mathbb{P}}[a_{\i_d}]=\bar{a}_{\i_d}$ and  $\mathbb{E}_{\mathbb{P}}[b_{\i_d}]=\bar{b}_{\i_d}$, then the resulting maximum-entropy distribution is $\mathbb{P}_\D$. If we drop the second set of constraints, then we get some other maximum-entropy distribution $\mathbb{P}_{\L}\neq\mathbb{P}_\D$ for ensemble $\L\neq\D$.
}
	\label{fig5}
\end{figure}

The fact that the second type of boundary-presence observables are also constrained in $\D(n;\p_1,\ldots,\p_n)$ may appear quite unexpected at first glance. The reason for the presence of these constraints is that simplex existence probabilities are actually conditional, where the conditions are the presence of simplex boundaries. If we go from conditional to unconditional probabilities, we change $\D$ to $\L$. Indeed, in $\D$, $p_{\i_d}$ is the conditional probability of the $(d-1)$-simplex $\{\i_d\}$ to appear in $C$, given that its $(d-2)$-dimensional boundary is already in place,
\begin{equation}\label{eq:condprob}
p_{\i_d}=\mathbb{P}_{\D}(a_{\i_d}=1|b_{\i_d}=1)
=\frac{\mathbb{P}_{\D}(a_{\i_d}=1,b_{\i_d}=1 )}{\mathbb{P}_{\D}(b_{\i_d}=1)}=\frac{\mathbb{P}_{\D}(a_{\i_d}=1)}{\mathbb{P}_{\D}(b_{\i_d}=1)},
\end{equation}
where the last equation follows from the compatibility condition $a_{\i_d}=1 \Rightarrow  b_{\i_d}=1$. This means that the unconditional probability of having $\{\i_d\}\in C$ is
\begin{equation}
\mathbb{P}_{\D}(a_{\i_d}=1)=p_{\i_d}\mathbb{P}_{\D}(b_{\i_d}=1),
\end{equation}
and, therefore, the expected values of observables $a_{\i_d}$ and $b_{\i_d}$ satisfy
\begin{equation}\label{eq:a=pb}
\bar{a}_{\i_d}=\mathbb{E}_{\D}[a_{\i_d}]=\mathbb{P}_{\D}(a_{\i_d}=1)=p_{\i_d}\mathbb{P}_{\D}(b_{\i_d}=1)=p_{\i_d}\bar{b}_{\i_d}.
\end{equation}
Thus, if we want to represent $\D$ as an ERSC and we fixed the expected values of the observables of the first type $a_{\i_d}$, we must also fix the expected values of the  observables of the second type $b_{\i_d}$. Moreover,  these expected values are not independent and must satisfy $\bar{a}_{\i_d}=p_{\i_d}\bar{b}_{\i_d}$, which is consistent with Proposition~\ref{prop:4}. In Appendix~\ref{sec:A4} we consider a special case with $n=3$, $p_{\i_1}=1$, $p_{\i_2}=p_1$, and $p_{\i_3}=p_2$, and explicitly show that the maximum-entropy distributions with and without the second type constraints are different.

To conclude, $\D\neq\L$.
From the maximum-entropy point of view, the ensemble $\L$, with only the observables of the first type constrained, appears more natural than $\D$. Yet $\D$ is more natural than $\L$ in terms of simplicity of its constructive Definition~\ref{def:fat-delta} that allows for efficient sampling of simplicial complexes. We leave open the questions of whether there exist ways to calculate the probability distribution $\mathbb{P}_{\L}(C)$ in ensemble $\L$, and to efficiently sample from it, i.e., to easily generate simplicial complexes $C$ with this probability.

\section*{Acknowledgments}
We thank  Alexandru Suciu, Gabor Lippner, and Christopher King for very useful discussions, comments, and suggestions.
This work was supported by DARPA grant No. HR0011-12-1-0012; NSF grants No. CNS-1344289, CNS-1442999,
CNS-0964236, CNS-1441828, CNS-1039646, and CNS-1345286; and by Cisco Systems.

\appendix
\renewcommand{\thesubsection}{A.\arabic{subsection}}
\renewcommand{\theequation}{A.\arabic{equation}}
\newpage
\section*{Appendix}

\subsection{Proof of Proposition~\ref{prop:3}}\label{sec:A1}
Let us first compute the expected number of $d$-simplices in $C^{[d]}$, where $C\sim\Delta(n;p_1,\ldots,p_{n-1})$.
\begin{equation}
\begin{split}
\bar{\phi}_d= &\mathbb{E}[\phi_d]=\mathbb{E}\left[\sum_{\i_{d+1}}b_{\i_{d+1}} \right]=\mathbb{E}\left[\sum_{\i_{d+1}}\prod_{k=1}^{d+1} a_{\i_{d+1}^{\hat{k}}} \right]\\
=&\left.
\mathbb{E}\left[\mathbb{E}\left[\sum_{\i_{d+1}}\prod_{k=1}^{d+1} a_{\i_{d+1}^{\hat{k}}}\right|\a_{d-1}\right]\right]=\mathbb{E}\left[\sum_{\i_{d+1}}\prod_{k=1}^{d+1} \left.\mathbb{E}\left[a_{\i_{d+1}^{\hat{k}}}\right|\a_{d-1}\right]\right].
\end{split}
\end{equation}
If the boundary of $(d-1)$-simplex $\{\i_{d+1}^{\hat{k}}\}$ belongs to $C$, \ie $ b_{\i_{d+1}^{\hat{k}}}=1$, then $\{\i_{d+1}^{\hat{k}}\}\in C$, \ie $a_{\i_{d+1}^{\hat{k}}}=1$, with probability $p_{d-1}$. Otherwise, if $ b_{\i_{d+1}^{\hat{k}}}=0$, then automatically $a_{\i_{d+1}^{\hat{k}}}=0$. Therefore, the inner expected value:
\begin{equation}
\left.\mathbb{E}\left[a_{\i_{d+1}^{\hat{k}}}\right|\a_{d-1}\right]
 =p_{d-1}b_{\i_{d+1}^{\hat{k}}}.
\end{equation}
So,
\begin{equation}
\bar{\phi}_d=\mathbb{E}\left[\sum_{\i_{d+1}}\prod_{k=1}^{d+1}p_{d-1}b_{\i_{d+1}^{\hat{k}}}\right]=p_{d-1}^{d+1}\mathbb{E}\left[\sum_{\i_{d+1}}\prod_{\k_2}a_{\i_{d+1}^{\widehat{\k}_2}}\right],
\end{equation}
where $\k_2=k_1,k_2$ is a pair of indices $1\leq k_1<k_2\leq d+1$, and $\i_{d+1}^{\widehat{\k}_2}=i_1,\ldots,\widehat{i_{k_1}},\ldots,\widehat{i_{k_2}},\ldots,i_{d+1}$ is the $(d-1)$-long multi-index with $i_{k_1}$ and $i_{k_2}$ omitted. Proceeding in this manner, we have:
\begin{equation}
\begin{split}
\bar{\phi}_d=&p_{d-1}^{d+1}\left.
\mathbb{E}\left[\mathbb{E}\left[\sum_{\i_{d+1}}\prod_{\k_2} a_{\i_{d+1}^{\widehat{\k}_2}}\right|\a_{d-2}\right]\right]=
p_{d-1}^{d+1}
\mathbb{E}\left[\sum_{\i_{d+1}}\prod_{\k_2}\left. \mathbb{E}\left[a_{\i_{d+1}^{\widehat{\k}_2}}\right|\a_{d-2}\right]\right]\\
=&p_{d-1}^{d+1}
\mathbb{E}\left[\sum_{\i_{d+1}}\prod_{\k_2}p_{d-2}b_{\i_{d+1}^{\widehat{\k}_2}}\right]=p_{d-1}^{{d+1\choose 1}}p_{d-2}^{{d+1\choose 2}}\mathbb{E}\left[\sum_{\i_{d+1}}\prod_{\k_3}a_{\i_{d+1}^{\widehat{\k}_3}}\right]=\ldots\\
=&p_{d-1}^{{d+1\choose 1}}\ldots p_{1}^{{d+1\choose d-1}}\mathbb{E}\left[\sum_{\i_{d+1}}\prod_{\k_d}a_{\i_{d+1}^{\widehat{\k}_d}}\right]={n \choose d+1}\prod_{k=1}^{d-1}p_k^{{d+1\choose d-k}}.
\end{split}
\end{equation}
The last equation holds because $a_{\i_{d+1}^{\widehat{\k}_d}}=1$ for any $\i_{d+1}$ and $\k_d$, since all simplicial complexes $C\sim\Delta(n;p_1,\ldots,p_{n-1})$ have exactly $n$ vertices. The expected number of $d$-simplices in $C^{(d)}$ is now:
\begin{equation}
\bar{f}_d=\mathbb{E}[f_d]=\mathbb{E}[\mathbb{E}[f_d|\phi_d]]=\mathbb{E}[p_d\phi_d]=p_d\bar{\phi}_d={n \choose d+1}\prod_{k=1}^{d}p_k^{{d+1\choose d-k}}.
\end{equation}

\subsection{Proof of Proposition~\ref{prop:4}}
\label{sec:A2}

Computations are similar to those in the previous section.
\begin{equation}
\begin{split}
\bar{b}_{\i_d}=&\mathbb{E}[b_{\i_d}]=\mathbb{E}\left[\prod_{k=1}^da_{\i_d^{\hat{k}}}\right]=\left.\mathbb{E}\left[\mathbb{E}\left[\prod_{k=1}^da_{\i_d^{\hat{k}}}\right|\a_{d-2}\right]\right]\\=&
\mathbb{E}\left[\prod_{k=1}^d\left.\mathbb{E}\left[a_{\i_d^{\hat{k}}}\right|\a_{d-2}\right]\right]=\mathbb{E}\left[\prod_{k=1}^dp_{\i_d^{\hat{k}}}b_{\i_d^{\hat{k}}}\right]=\prod_{k=1}^dp_{\i_d^{\hat{k}}}\mathbb{E}\left[\prod_{k=1}^db_{\i_d^{\hat{k}}}\right]\\
=&\prod_{k=1}^dp_{\i_d^{\hat{k}}}\mathbb{E}\left[\prod_{\k_2}a_{\i_d^{\widehat{\k}_2}}\right]=\prod_{k=1}^dp_{\i_d^{\hat{k}}}\prod_{\k_2}p_{\i_d^{\widehat{\k}_2}}\mathbb{E}\left[\prod_{\k_2}^db_{\i_d^{\widehat{\k}_2}}\right]=\ldots\\
=&\prod_{k=1}^dp_{\i_d^{\hat{k}}}\ldots\prod_{\k_{d-1}}p_{\i_d^{\widehat{\k}_{d-1}}}\mathbb{E}\left[\prod_{\k_{d-1}}^db_{\i_d^{\widehat{\k}_{d-1}}}\right]=\prod_{m=1}^{d-1}\prod_{\k_m}p_{\i_d^{\widehat{\k}_m}},
\end{split}
\end{equation}
since $b_{\i_d^{\widehat{\k}_{d-1}}}=1$ for any $\i_d$ and $\k_{d-1}$. Finally,
\begin{equation}
\bar{a}_{\i_d}=\mathbb{E}[{a}_{\i_d}]=\mathbb{E}[\mathbb{E}[{a}_{\i_d}|{b}_{\i_d}]]=\mathbb{E}[p_{\i_d}b_{\i_d}]=\prod_{m=0}^{d-1}\prod_{\k_m}p_{\i_d^{\widehat{\k}_m}}.
\end{equation}

 \subsection{Special case: $\Delta(3;p_1,p_2)$.}\label{sec:A3}
Theorem~\ref{thm:Kahle} in Section~\ref{sec:Khale} explicitly represents the Kahle's multi-parameter model of random simplicial complexes $\Delta(n;p_1,\ldots,p_{n-1})$ as an ERSC for any values of the parameters. This theorem is a direct corollary of Lemmas~\ref{lemma1}\&\ref{lemma2} in  Section~\ref{sec:maximum_entropy} which assert that any distribution is, in fact, the maximum-entropy distribution under certain constraints. Here we illustrate the difficulties that arise when one tries to compute the maximum-entropy distribution $\mathbb{P}_\Delta$ using Theorem~\ref{thm1}. We successfully used this method, which is based on computing the partition function, in Section~\ref{sec:simple examples} for the Erd\H{o}s--R\'{e}nyi random graphs and the random flag and Linial--Meshulam complexes. For Kahle's $\Delta$-ensemble, however, the partition function becomes intractable.

Consider a special case of the Kahle's model with $n=3$.  According to Theorem~\ref{thm:Kahle} and Proposition~\ref{prop:3}, $\Delta(3;p_1,p_2)$ is the maximum-entropy ensemble of simplicial complexes on $3$ vertices with three constraints:
\begin{equation}
	\label{eq:APPconstraints}
	\mathbb{E}[f_1]=3p_1, \hspace{5mm} \mathbb{E}[f_2]=p_1^3p_2, \hspace{5mm} \mathbb{E}[{\phi}_2]=p_1^3.
\end{equation}
Let us compute the corresponding maximum-entropy distribution $\mathbb{P}_{\Delta(3;p_1,p_2)}$ using Theorem~\ref{thm1}. The partition function $Z$ in (\ref{eq:ERSC}) is
\begin{equation}
\begin{split}
Z(\theta_1,\theta_2,\theta_3)=&\sum_{C\in\mathcal{C}_3} e^{-H(C)}=\sum_{C\in\mathcal{C}_3} e^{-\theta_1f_1(C)-\theta_2f_2(C)-\theta_3\phi_2(C)}\\
=&1+3e^{-\theta_1}+3e^{-2\theta_1}+e^{-3\theta_1-\theta_3}+e^{-3\theta_1-\theta_2-\theta_3},
\end{split}
\end{equation}
where the last equality follows from Figure~\ref{fig6}, where we list all complexes in $\mathcal{C}_3$ along with the corresponding values of observables $f_1, f_2,$ and $\phi_2$. To find parameters $\theta_1$, $\theta_2$, and $\theta_3$, which are the Lagrange multipliers coupled to observables $f_1$, $f_2$, and $\phi_2$, we need to solve the system of three equations (\ref{eq:parametersC}), where $\bar{x}_i$ are replaced by the expected values in (\ref{eq:APPconstraints}):
\begin{align}
\frac{3e^{-\theta_1}+6e^{-2\theta_1}+3e^{-3\theta_1}e^{-\theta_3}+3e^{-3\theta_1}e^{-\theta_2}e^{-\theta_2}}{1+3e^{-\theta_1}+3e^{-2\theta_1}+e^{-3\theta_1}e^{-\theta_3}+e^{-3\theta_1}e^{-\theta_2}e^{-\theta_3}} &=3p_1, \nonumber \\
\frac{e^{-3\theta_1}e^{-\theta_2}e^{-\theta_3}}{1+3e^{-\theta_1}+3e^{-2\theta_1}+e^{-3\theta_1}e^{-\theta_3}+e^{-3\theta_1}e^{-\theta_2}e^{-\theta_3}}&=p_1^3p_2, \label{horror}\\
\frac{e^{-3\theta_1}e^{-\theta_3}+e^{-3\theta_1}e^{-\theta_2}e^{-\theta_3}}{1+3e^{-\theta_1}+3e^{-2\theta_1}+e^{-3\theta_1}e^{-\theta_3}+e^{-3\theta_1}e^{-\theta_2}e^{-\theta_3}}&=p_1^3. \nonumber
\end{align}
After some tedious algebra, one can show that the solution is
\begin{equation}
e^{-\theta_1}=\frac{p_1}{1-p_1}, \hspace{5mm} e^{-\theta_2}=\frac{p_2}{1-p_2}, \hspace{5mm}
e^{-\theta_3}=1-p_2.
\end{equation}
The partition function simplifies then to
\begin{equation}
Z=\frac{1}{(1-p_1)^3}.
\end{equation}
Therefore, the maximum-entropy distribution is
\begin{equation}\label{P_n=3}
\begin{split}
\mathbb{P}_{\Delta(3;p_1,p_2)}(C)=&\frac{e^{-H(C)}}{Z}=
\frac{e^{-\theta_1f_1(C)-\theta_2f_2(C)-\theta_3\phi_2(C)}}{Z}\\
=&(1-p_1)^3\left(\frac{p_1}{1-p_1}\right)^{f_1(C)}\left(\frac{p_2}{1-p_2}\right)^{f_2(C)}\left(1-p_2\right)^{\phi_2(C)}\\
=&p_1^{f_1(C)}(1-p_1)^{3-f_1(C)}p_2^{f_2(C)}(1-p_2)^{\phi_2(C)-f_1(C)}.
\end{split}
\end{equation}
As expected, the obtained distribution coincides with the distribution in (\ref{eq:Delta}), where $n=3$ and $\phi_1(C)=3$. Unfortunately, this method of computing $\mathbb{P}_\Delta$ cannot be extended to the general case $\Delta(n;p_1,\ldots,p_{n-1})$: when $n>3$ the partition function $Z$ and the corresponding analog of system (\ref{horror}) become analytically intractable. This makes Lemmas~\ref{lemma1}\&\ref{lemma2} an essential tool for proving Theorem~\ref{thm:Kahle} and a more general Theorem~\ref{thm:3}.

\begin{figure}
	\centerline{\includegraphics[width=80mm]{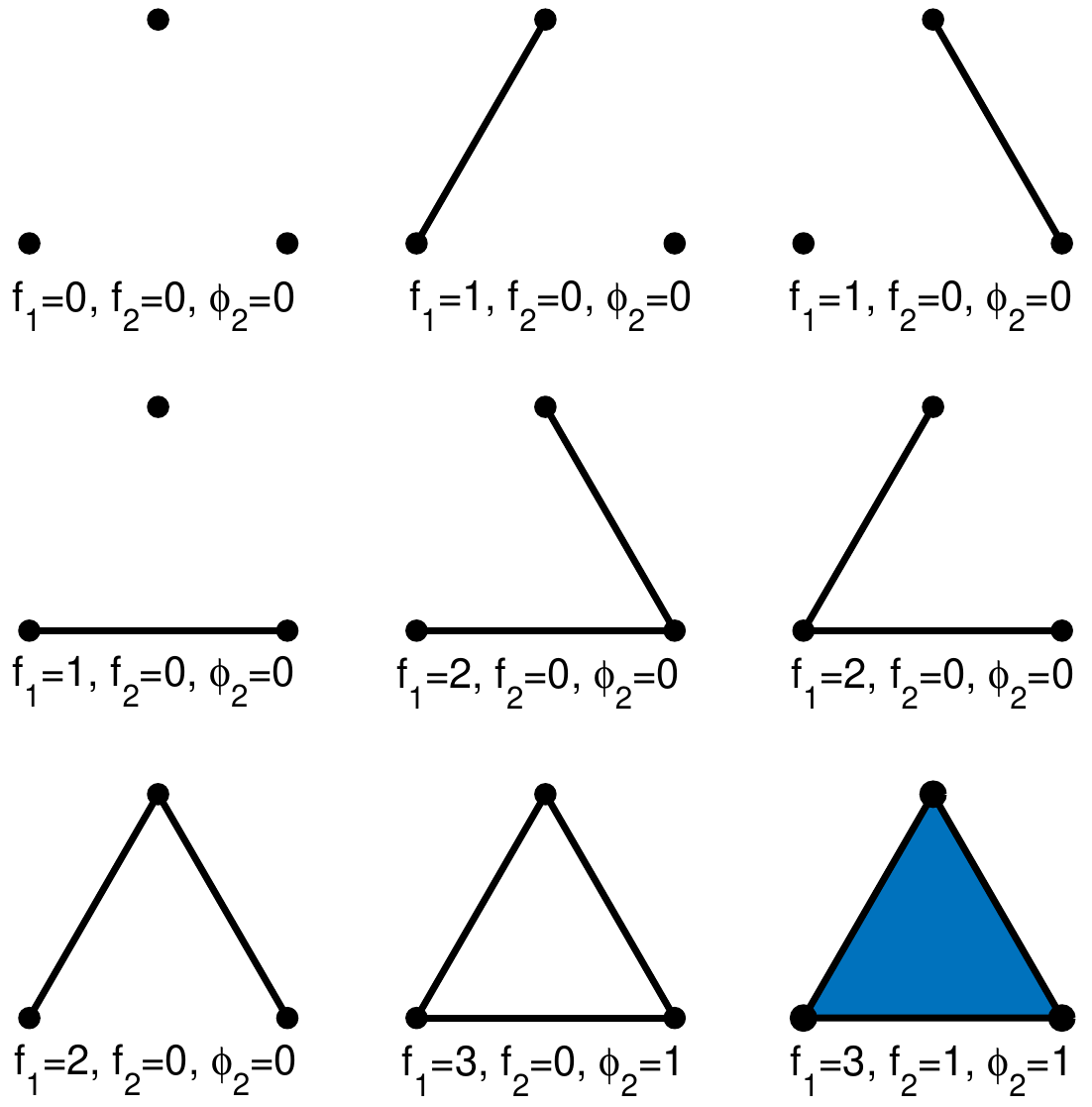}}
	\caption{\small \textbf{Simplicial complexes on three vertices.} Here we show all $C\in\mathcal{C}_3$ and the values of $f_1$ (number of 1-simplices), $f_2$ (number of 2-simplices), and $\phi_2$ (number of 2-simplices in $C^{[2]}$) for each $C$. }
	\label{fig6}
\end{figure}

\subsection{ERSC($\mathcal{C}_3,\{f_1,f_2\},\{\bar{f}_1,\bar{f}_2\}$)}\label{sec:A4}
Here we derive the maximum-entropy distribution on $\mathcal{C}_3$ only under the constraints of the first type, $\mathbb{E}[f_1]=\bar{f}_1$ and $\mathbb{E}[f_2]=\bar{f}_2$, and show that it is different from $\mathbb{P}_{\Delta(3,p_1,p_2)}$. This explicitly demonstrates that the constraint of the second type, $\mathbb{E}[\phi_2]=\bar{\phi}_2$, is not redundant, and, if dropped, the resulting maximum-entropy ensemble will no longer be $\Delta$.

Let $(\mathcal{C}_3,\widetilde{\mathbb{P}})$ be  the maximum-entropy ensemble  ERSC($\mathcal{C}_3,\{f_1,f_2\},\{\bar{f}_1,\bar{f}_2\}$). In other words, $\widetilde{\mathbb{P}}$ is the maximum-entropy distribution on $\mathcal{C}_3$ under the constraints
\begin{equation}
\mathbb{E}[f_1]=\bar{f}_1 \hspace{5mm} \mbox{and}  \hspace{5mm} \mathbb{E}[f_2]=\bar{f}_2.
\end{equation}
We can find $\widetilde{\mathbb{P}}$ using Theorem~\ref{thm1} as in the previous section. The partition function
\begin{equation}
\begin{split}
\widetilde{Z}(\theta_1,\theta_2)=&\sum_{C\in\mathcal{C}_3}e^{-\widetilde{H}(C)}=\sum_{C\in\mathcal{C}_3}e^{-\theta_1f_1(C)-\theta_2f_2(C)}\\
=& \left(1+e^{-\theta_1}\right)^3+e^{-3\theta_1}e^{-\theta_2},
\end{split}
\end{equation}
where the last equality is obtained with the help of Figure~\ref{fig6}. The system of equations (\ref{eq:parametersC}) for $\theta_1$ and $\theta_2$ is then
\begin{equation}
\begin{split}
3\frac{e^{-\theta_1}\left(1+e^{-\theta_1}\right)^2+e^{-3\theta_1}e^{-\theta_2}}{\left(1+e^{-\theta_1}\right)^3+e^{-3\theta_1}e^{-\theta_2}}=& \bar{f}_1,\\
\frac{e^{-3\theta_1}e^{-\theta_2}}{\left(1+e^{-\theta_1}\right)^3+e^{-3\theta_1}e^{-\theta_2}}=& \bar{f}_2,
\end{split}
\end{equation}
and one can check that the solution is given by
\begin{equation}
e^{-\theta_1}=\frac{\frac{\bar{f}_1}{3}-\bar{f}_2}{1-\frac{\bar{f}_1}{3}} \hspace{5mm} \mbox{and}  \hspace{5mm} e^{-\theta_2}=\frac{\bar{f}_2(1-\bar{f}_2)^2}{\left(\frac{\bar{f}_1}{3}-\bar{f}_2\right)^3}.
\end{equation}
The partition function, as a function of $\bar{f}_1$ and $\bar{f}_2$, is then
\begin{equation}
\widetilde{Z}=\frac{\left(1-\bar{f}_2\right)^2}{\left(1-\frac{\bar{f}_1}{3}\right)^3}.
\end{equation}
Therefore,  the maximum-entropy distribution is
\begin{equation}
\begin{split}
\widetilde{\mathbb{P}}=&\frac{e^{-\widetilde{H}(C)}}{\widetilde{Z}}=\frac{e^{-\theta_1f_1(C)-\theta_2f_2(C)}}{\widetilde{Z}}\\
=& \frac{\left(1-\frac{\bar{f}_1}{3}\right)^3}{\left(1-\bar{f}_2\right)^2}\left(\frac{\frac{\bar{f}_1}{3}-\bar{f}_2}{1-\frac{\bar{f}_1}{3}}\right)^{f_1(C)}\left(\frac{\bar{f}_2(1-\bar{f}_2)^2}{\left(\frac{\bar{f}_1}{3}-\bar{f}_2\right)^3}\right)^{f_2(C)}\\
=&\left(\frac{\bar{f}_1}{3}-\bar{f}_2\right)^{f_1(C)-3f_2(C)}\left(1-\frac{\bar{f}_1}{3}\right)^{3-f_1(C)}\bar{f}_2^{f_2(C)}\left(1-\bar{f}_2\right)^{2f_2(C)-2}.
\end{split}
\end{equation}
This is a general expression for $\widetilde{\mathbb{P}}$ for any expected values $\bar{f}_1$ and $\bar{f}_2$. In the special case, when $\bar{f}_1$ and $\bar{f}_2$ coincide with the corresponding values for $\Delta(3;p_1,p_2)$ in (\ref{eq:APPconstraints}), that is $\bar{f}_1=3p_1$ and $\bar{f}_2=p_1^3p_2$, the distribution $\widetilde{\mathbb{P}}$ reduces to
\begin{equation}
\widetilde{\mathbb{P}}=p_1^{f_1(C)}(1-p_1)^{3-f_1(C)}p_2^{f_2(C)}(1-p_1^2p_2)^{f_1(C)-3f_2(C)}(1-p_1^3p_2)^{2f_2(C)-2}.
\end{equation}
We see that $\widetilde{\mathbb{P}}\neq\mathbb{P}_{\Delta(3;p_1,p2)}$. This means that the two maximum-entropy ensembles $\Delta(3;p_1,p_2)$ and $(\mathcal{C}_3,\widetilde{\mathbb{P}})$ are different,
\begin{equation}
\mathrm{ERSC}(\mathcal{C}_3,\{f_1,f_2,\phi_2\},\{\bar{f}_1,\bar{f}_2,\bar{\phi}_2\})\neq\mathrm{ERSC}(\mathcal{C}_3,\{f_1,f_2\},\{\bar{f}_1,\bar{f}_2\}),
\end{equation}
and, more generally,
\begin{equation}
\mathrm{ERSC}\left(\mathcal{C}_n,\left\{\{f_d\}_{d=1}^{n-1},\{\phi_d\}_{d=2}^{n-1}\right\},\left\{\{\bar{f}_d\}_{d=1}^{n-1},\{\bar{\phi}_d\}_{d=2}^{n-1}\right\}\right)\neq\mathrm{ERSC}\left(\mathcal{C}_n,\{f_d\}_{d=1}^{n-1},\{\bar{f}_d\}_{d=1}^{n-1}\right).
\end{equation}

\bibliography{bib}
\end{document}